[1994/12/01]
\documentclass{amsart}
\chardef\bslash=`\\ 

\usepackage{amsthm,amsmath,amssymb, color, graphics}
\usepackage{tikz-cd}
\usepackage{mathrsfs}
\usepackage{adjustbox}
\usepackage{hyperref}
 \usepackage{amsaddr}

\newtheorem{thm}{Theorem}[section]
\newtheorem{claim}[thm]{Claim}
\newtheorem{cor}[thm]{Corollary}

\newtheorem{lem}[thm]{Lemma}
\newtheorem{prop}[thm]{Proposition}

\newtheorem{question}[thm]{Question}
\theoremstyle{definition}
\newtheorem{defn}[thm]{Definition}
\newtheorem{rem}[thm]{Remark}
\newtheorem{exa}[thm]{Example}

\newtheorem{fact}[thm]{Fact}

\newcommand{\Ab}{\mathrm{Ab}}
\newcommand\ZZ{\mathbb{Z}}

\newcommand\G{\mathcal{G}}

\newcommand{\limn}[1]{\operatorname{lim}^{#1}}

\newcommand\xbf{\mathbf{x}}

\newcommand\bb[1]{^{[\![#1]\!]}}

\DeclareMathOperator{\colim}{colim}

\newcommand\PH{\mathrm{PH}}
\newcommand\dom{\operatorname{dom}}
\newcommand{\Om}{\Omega}

\begin{document}

\title[Additivity of Derived Limits]{Addditivity of Derived Limits in the Cohen model}
\author{Nathaniel Bannister}
\address{Department of Mathematical Sciences \\
Carnegie Mellon University \\
Pittsburgh, PA 15213}
\email{nathanib@andrew.cmu.edu}

\date{Received on MONTH, YEAR}
\issueinfo{VOL}{NUM}{MONTH}{YEAR}
\begin{abstract}
    We show that, in the model constructed by adding sufficiently many Cohen reals, derived limits are additive on a large class of systems. This generalizes the work of Jeffrey Bergfalk, Michael Hru\v s\'ak, and Chris Lambie-Hanson which focuses on the system $\mathbf{A}$.
    In the process, we isolate a partition principle responsible for the vanishing of derived limits on collections of Cohen reals and reframe the propagating trivializations results of Bergfalk, Hru\v s\'ak and Lambie-Hanson as a theorem of ZFC.
    In light of results of the author, Jeffrey Bergfalk, and Justin Moore, the additivity of derived limits also implies additivity results for strong homology. 
\end{abstract}

\maketitle
\thanks{
The author would like to thank James Cummings for helpful discussions and considerable aid in revising this paper, as well as Jeffrey Bergfalk for insightful discussions and comments regarding this work, and Matteo Casarosa for interest in the project and encouragement to complete it. 
The author would also like to thank the anonymous referee for helpful comments and suggestions.
}
\tableofcontents
\section{Introduction}

Strong homology was originally defined by Marde\v{s}i\'c and offers a correction to the theory of \v{C}ech homology, which fails to satisfy the exactness axiom. 
A comprehensive reference for strong homology is \cite{SSH}. 
Motivated by questions coming from strong homology, in particular, whether it is \emph{additive} in the sense that it maps disjoint sums of spaces to direct sums of homology groups, Marde\v{s}i\'c and Prasolov consider in \cite{SHINA} the inverse system $\mathbf{A}$ of abelian groups. 
$\mathbf{A}$ is indexed over $\omega^\omega$ (the set of all functions from the natural numbers to the natural numbers) with respect to the ordering $x\leq y$ if and only if for all $n$, $x(n)\leq y(n)$; for each index $x$ the group $\mathbf{A}_x$ is defined by
\[\mathbf{A}_x=\bigoplus_{i<\omega}\ZZ^{x(i)}\]
and if $x\leq y$ the map from $\mathbf{A}_y$ to $\mathbf{A}_x$ is the canonical projection map.
The system $\mathbf{A}$ appears naturally in the computation of the strong homology of a countable disjoint sum of Hawaiian earrings, and of higher dimensional variants of the Hawaiian earrings.
Marde\v{s}i\'c and Prasolov show that strong homology is additive on a countable disjoint sum of Hawaiian earrings if and only if the derived limits $\lim^n\mathbf{A}$ are $0$ for every $n>0$. 
Moreover, they connect the vanishing of $\lim^1\mathbf{A}$ to a set theoretic question of the form ``is every coherent family trivial?'' and show that, in the presence of the Continuum Hypothesis, $\lim^1\mathbf{A}\neq0$. 
Several years later, Dow, Simon, and Vaughan show in \cite{DSV} that $\lim^1\mathbf{A}=0$ in the presence of the Proper Forcing Axiom, a set theoretic assumption widely believed to have the consistency strength of a supercompact cardinal.
In \cite[Theorem 8.7]{PPT}, Todor\v{c}evi\'c reduces the hypothesis to the Open Graph Axiom (also referred to as , a consequence of the Proper Forcing Axiom which has no large cardinal strength. 

Prasolov proves in \cite{NASH} that in ZFC set theory strong homology is not additive on the class of all topological spaces. 
The example involves a countable disjoint sum of copies of a Hawaiian earring with uncountably many circles, and nonadditivity is established by showing that the first derived limit of a ``tall'' version of $\mathbf{A}$ is nonzero. 
Prasolov's result still leaves the question of whether there can be a rich class of spaces on which strong homology is additive - the spaces used in the example are not metrizable, nor even first countable. 

In recent years, attention has returned to the derived limits of $\mathbf{A}$ as well as to the motivating question of the classes of spaces on which strong homology can consistently be additive. 
\begin{itemize}
    \item Bergfalk shows in \cite{SHDLST} that the Proper Forcing Axiom implies that $\lim^2\mathbf{A}\neq0$; in particular, the Proper Forcing Axiom still implies a failure of additivity for strong homology even for the class of closed subspaces of $\mathbb{R}^n$. 
    \item In a recent paper, Veli\v{c}kovi\'c and Vignati (\cite{NVHDL}) prove that for any $n$, it is consistent relative to ZFC that $\lim^n\mathbf{A}$ is nonzero. 
    \item In \cite{SVHDL}, Bergfalk and Lambie-Hanson show that consistently from a mild large cardinal hypothesis, $\lim^n\mathbf{A}=0$ for all $n>0$. 
    As a corollary, in their model strong homology is additive on a countable sum of Hawaiian earrings. 
    \item The author, joint with Bergfalk and Moore, shows in \cite{ASH} that in the model of \cite{SVHDL}, strong homology is additive on the class of locally compact separable metric spaces. 
    This result involves generalizing the definition of the system $\mathbf{A}$ to a class of inverse systems of abelian groups known as ``$\Om$ systems,'' showing that additivity of strong homology for locally compact separable metric spaces is equivalent to the additivity of derived limits of $\Om$ systems, and finally showing that additivity of derived limits for $\Om$ systems holds in the model of \cite{SVHDL}. 
    \item In \cite{SVHDLwLC}, Bergfalk, Hru\v{s}\'ak, and Lambie-Hanson remove the large cardinal hypothesis for the vanishing of the derived limits of $\mathbf{A}$. 
\end{itemize}

The main theorem of this paper is Theorem \ref{add} that in the same model used by Bergfalk, Hru\v{s}\'ak, and Lambie-Hanson, derived limits are additive for the class of all $\Om$ systems; in light of \cite[Theorem 6]{ASH}, strong homology is additive on the class of locally compact separable metric spaces. 
The main corollary is the following:

\begin{cor}
    It is consistent relative to ZFC that strong homology is additive on the class of locally compact separable metric spaces. 
\end{cor}

The key new ingredient in the proof is to analyze the injective objects in the abelian category of $\Om$ systems, which we call the {\em{injective systems}}; additional properties we will prove about injective systems will allow us to mimic the proofs in \cite{SVHDLwLC}. 
The class of injective systems is rich enough that vanishing of higher derived limits on injective systems implies additivity of derived limits for all $\Om$ systems. 

The techniques involved to produce these results naturally apply to a broader class of systems, which we call ``$\Om_\kappa$ systems.'' 
For infinite cardinals $\kappa$, the $\Om_\kappa$ systems are a ``wide'' version of $\Om$ systems, where we now allow for $\kappa$ many towers. 
The associated objects include the system called $\mathbf{A}_\kappa$ in \cite{SHDLST}, the natural variation of $\mathbf{A}$ indexed over $\omega^\kappa$.
The main theorem we will prove is the following:
\begin{thm} \label{add}
Suppose that $\chi\geq\beth_\omega(\kappa)$. 
Then $\operatorname{Add}(\omega,\chi)$, the forcing to add $\chi$ many Cohen reals, forces that whenever $\mathcal{G}$ is an $\Om_{\kappa}$ system, the canonical map 
\[\bigoplus_{\alpha<\kappa}\lim{}^n\G_\alpha\to \lim{}^nG\]
is an isomorphism. 
\end{thm}
Additivity of derived limits for $\Om_\kappa$ systems holds implications for strong homology as well. 
For example, we show in section \ref{kappa_add_sh} the following generalization of \cite[Theorem 6]{ASH}:
\begin{thm} \label{add_kappa}
    The following are equivalent:
    \begin{enumerate}
\item whenever $X$ is a locally compact metric space of weight at most $\kappa$, 
\[\overline{H}_p(X)\cong\underset{\substack{K\subseteq X\\K\text{ compact}}}{\colim}\overline{H}_p(K),\]
where $\overline{H}_p$ is strong homology (recall that the weight of a topological space is the minimum cardinality of a basis).
That is, strong homology has \emph{compact supports} on the class of locally compact separable metric spaces. \label{cs}
    \item whenever $\langle X_i\mid i<\kappa\rangle$ are locally compact metric spaces of weight at most $\kappa$, the natural map
\[\bigoplus_{i<\kappa}\overline{H}_p(X_i)\to\overline{H}_p\left(\coprod_{i<\kappa }X_i\right)\]
is an isomorphism; that is, strong homology is additive on the class of locally compact metric spaces of weight at most $\kappa$. \label{add_lc}
\item whenever $\langle X_i\mid i<\kappa\rangle$ are compact metric spaces, the natural map 
\[\bigoplus_{i<\kappa}\overline{H}_p(X_i)\to\overline{H}_p\left(\coprod_{i<\kappa }X_i\right)\]
is an isomorphism. \label{add_c}
        \item whenever $\mathcal{G}$ is an $\Om_{\kappa}$ system with all groups finitely generated, the canonical map 
\[\bigoplus_{\alpha<\kappa}\lim{}^n\G_\alpha\to \lim{}^nG\]
is an isomorphism. \label{add_om}
    \end{enumerate}
\end{thm}

The key corollary of our results is then
\begin{cor}
    Suppose that $\chi\geq\beth_\omega(\kappa)$. Then $\operatorname{Add}(\omega,\chi)$ forces that strong homology is additive and has compact supports on the class of locally compact metric spaces of weight at most $\kappa$. 
\end{cor}

The structure of the article is as follows: in section 2, we present our basic conventions and the relevant facts about higher dimensional $\Delta$-systems required for the argument. 
The specific definition of higher dimensional $\Delta$-systems we use was isolated by Lambie-Hanson in \cite{HDDS}. 
In section 3, we define $\Om_\kappa$ systems and explore some of their basic properties. 
We analyze the injective objects and define the notions of $n$-coherence and $n$-triviality to conclude that if every $n$-coherent family corresponding to every $\Om_\kappa$ system is trivial then derived limits are additive on the class of $\Om_\kappa$ systems and strong homology is additive on the class of locally compact separable metric spaces of weight at most $\kappa$.
Our definition of coherence generalizes the notion of higher coherence defined in \cite[Definition 3.1]{SHDLST}; the notion of triviality used there coincides with our definition of type $I$ trivial. 
Type $II$ triviality corresponds to the notion of triviality introduced in \cite[Proposition 2.13]{SVHDL}. 
Section 4 begins our analysis of how to generalize the results of \cite{SVHDLwLC} to the context of injective systems.
We start by generalizing the arguments for \cite[Lemma 4.7]{SVHDLwLC} to injective $\Om_\kappa$ systems, demonstrating how to propagate a trivialization for an $n$-coherent family from a sufficiently large set of sequences to all sequences.  We rephrase the result as a consequence of $\mathrm{ZFC}$ rather than specifically a result about Cohen forcing.
In section 5, we define an analogue of the Ramsey theoretic assertion $\PH_n$ from \cite{DAHDL} and prove that it holds after adding sufficiently many Cohen sequences. 
We complete the argument for additivity of derived limits in section 6 by showing how the partition principle implies that coherent families are trivial on a set of sequences which is sufficiently large for the propagation results of section 4. 
Finally, we note a number of questions still remaining open. 

In section 4, we assume familiarity with the basics of Cohen's method of forcing, as can be found in \cite[Chapter 4]{Kunen}. 
The proofs in section 3 rely on tools from homological algebra, particularly the theory of derived functors and the role injective objects play in this theory; most of the statements themselves should be readable without this background. 
A good reference for the relevant material is \cite[Chapters 1-3]{Weibel}.  

\section{Preliminaries} \label{Prelim}
\subsection{Notational conventions}
If $X$ is a set and $\kappa$ is a cardinal, then $[X]^\kappa=\{Y\subseteq X\,\mid\,|Y|=\kappa\}$ and $[X]^{<\kappa}=\{Y\subseteq X\mid |Y|<\kappa\}$. 
If $\kappa$ and $\lambda$ are cardinals, then we say that $\lambda$ is \emph{$<\!\kappa$-inaccessible} if $\nu^{<\kappa}<\lambda$ for all $\nu<\lambda$. 
If $X$ is a set of ordinals, then $\operatorname{otp}(X)$ denotes the order-type of $X$. 
We will often view finite sets of ordinals as finite increasing sequences of ordinals, and vice versa. 
For example, if $a \in[\mathrm{On}]^{<\omega}$ and $\ell<\operatorname{otp}(a)$, then $a(\ell)$ is the unique $\alpha \in a$ such that $|a \cap \alpha|=\ell .$ If $\mathbf{m} \subseteq \operatorname{otp}(a)$, then $a[\mathbf{m}]=\{a(\ell) \mid \ell \in \mathbf{m}\} .$ 
For any set $X$ of ordinals and natural number $n$, the notation $\left(\alpha_{0}, \ldots, \alpha_{n-1}\right) \in[X]^{n}$ will denote the conjunction of the statements $\left\{\alpha_{0}, \ldots, \alpha_{n-1}\right\} \in[X]^{n}$ and $\alpha_{0}<\ldots<\alpha_{n-1} .$ 
We write $a\sqsupseteq b$ to mean that $a$ is an end extension of $b$. 
For a sequence $\mathbf{x}$, we write $\mathbf{x}^i$ as the subsequence obtained by removing the element at index $i$. 

The forcings appearing in this paper will all be of the form $\mathbb{P}=\operatorname{Fn}(\chi\times\kappa,\omega)$ from \cite[Chapter 4, section 5]{Kunen}, where $\chi$ is an uncountable cardinal and $\chi\geq\kappa$ for some infinite cardinal $\kappa$. 
The poset $\mathbb{P}$ consists of finite partial functions from $\chi \times \kappa$ to $\omega$ ordered by reverse inclusion; note that $\mathbb{P}$ is isomorphic to the forcing to add $\chi$ many Cohen reals. 
Forcing with $\mathbb{P}$ produces a generic function $F: \chi \times \kappa \rightarrow \omega .$ For a fixed $\alpha<\chi$ we will refer to the function $F(\alpha, \cdot): \kappa \rightarrow \omega$ as the $\alpha^{\text {th }}$ Cohen sequence added by $\mathbb{P}$, and we will typically denote this function by $f_{\alpha} ;$ we denote the canonical $\mathbb{P}$-name in $V$ for $f_{\alpha}$ by $\dot{f}_{\alpha}$. 
If $G$ is $\mathbb{P}$-generic over $V$ and $W \subseteq \chi$ is in $V$ then $G_{W}$ denotes $\{p \in G \mid \operatorname{dom}(p) \subseteq W \times \kappa\} .$ 
For any condition $p$ in $\mathbb{P}$ let $u(p)$ denote the set $\{\alpha<\chi \mid \operatorname{dom}(p) \cap(\{\alpha\} \times \kappa) \neq \emptyset\}$ and let $\bar{p}$ denote the finite partial function from $\operatorname{otp}(u(p)) \times \kappa$ to $\omega$ defined as follows: for all $i<\operatorname{otp}(u(p))$ and all $\alpha<\kappa$, define $(i, \alpha)$ to be in the domain of $\bar{p}$ if and only if $(u(p)(i), \alpha) \in \operatorname{dom}(p)$; if so, let $\bar{p}(i, \alpha)=p(u(p)(i), \alpha) .$ 
Intuitively, $\bar{p}$ is a ``collapsed'' version of $p$ in the $\chi$ coordinate.  
Notice that the set $\{\bar{p} \mid p \in \mathbb{P}\}$ is a subset of the set of finite partial functions from $\omega \times \kappa$ to $\omega$.

Following standard conventions from homological algebra, chain complexes are denoted with a superscript bullet, replaced by an index to denote each group, so that $K^n$ is the $n$th group in the cochain complex $K^\bullet$.
$H^n(K^\bullet)$ is the $n$th cohomology group of the cochain complex $K^\bullet$.

\subsection{Higher Dimensional $\Delta$-systems}
Higher dimensional analogues of the classical notion of $\Delta$-systems began with the work of Todor\v{c}evi\'c \cite{RPPR} and Shelah \cite{WSR} in the '80s; the exact formulation we will use was isolated by Lambie-Hanson in \cite{HDDS}. 
In recent years, they have proven to be a useful tool, particularly for passing partition results from the ground model to forcing extensions (see for instance \cite{Jthesis, SVHDL, SVHDLwLC, HDDS, DAHDL, MS}). 
Here, we use them for the purpose of obtaining an analogue of the Ramsey theoretic assertion known as $\textrm{PH}_n$ from \cite{DAHDL}. 
Defining this variation of $\PH_n$, and using higher dimensional $\Delta$-systems to prove it holds after adding sufficiently many Cohen reals, is the topic of section 5.
We start with a preliminary definition; note that we identify sets of ordinals with their increasing enumerations.

\begin{defn}
Suppose that $a$ and $b$ are sets of ordinals.
\begin{itemize}
    \item We say that $a$ and $b$ are \emph{aligned} if $\operatorname{otp}(a)=\operatorname{otp}(b)$ and $\operatorname{otp}(a \cap \gamma)=\operatorname{otp}(b \cap \gamma)$ for all $\gamma \in a \cap b .$ In other words, if $\gamma$ is a common element of two aligned sets $a$ and $b$, then it occupies the same relative position in both $a$ and $b$.
    \item If $a$ and $b$ are aligned then we let $\mathbf{r}(a, b):=\{i<\operatorname{otp}(a) \mid a(i)=b(i)\} .$ Notice that, in this case, $a \cap b=a[\mathbf{r}(a, b)]=b[\mathbf{r}(a, b)]$
\end{itemize}
\end{defn}

The following definition of a higher dimensional $\Delta$-system is due to Lambie-Hanson (\cite{HDDS}).

\begin{defn}
Suppose that $H$ is a set of ordinals, $n$ is a positive integer, and $u_{b}$ is a set of ordinals for all $b \in[H]^{n} .$ We call $\left\langle u_{b} \mid b \in[H]^{n}\right\rangle$ a {\em{uniform $n$-dimensional $\Delta$-system}} if there exists an ordinal $\rho$ and, for each $\mathbf{m} \subseteq n$, a set $\mathbf{r}_{\mathbf{m}} \subseteq \rho$ satisfying the following statements.
\begin{enumerate}
    \item $\operatorname{otp}\left(u_{b}\right)=\rho$ for all $b \in[H]^{n}$.
    \item For all $a, b \in[H]^{n}$ and $\mathbf{m} \subseteq n$, if $a$ and $b$ are aligned with $\mathbf{r}(a, b)=\mathbf{m}$, then $u_{a}$ and $u_{b}$ are aligned with $\mathbf{r}\left(u_{a}, u_{b}\right)=\mathbf{r}_{\mathbf{m}}$.
    \item For all $\mathbf{m}_{0}, \mathbf{m}_{1} \subseteq n$, we have $\mathbf{r}_{\mathbf{m}_{0} \cap \mathbf{m}_{1}}=\mathbf{r}_{\mathbf{m}_{0}} \cap \mathbf{r}_{\mathbf{m}_{1}}$. In particular, if $\mathbf{m}_0\subseteq\mathbf{m}_1$ then $\mathbf{r}_{\mathbf{m}_0}\subseteq\mathbf{r}_{\mathbf{m}_1}$.
\end{enumerate}
\end{defn}

Definition \ref{k_add_def} and a special case of Lemma \ref{k_addable} appear as \cite[Claim 6.2/6.3]{SVHDLwLC}. 
For completeness, we will give the proof of Lemma \ref{k_addable}, which is essentially the same as what appears as \cite[Claim 6.2/6.3]{SVHDLwLC} applied to all higher dimensional $\Delta$-systems.

\begin{defn} \label{k_add_def}
$\alpha\in H$ is \emph{$k$-addable} for $a\in [H]^m$ if $\alpha\not\in a$ and $|a\cap \alpha|=k$. 
This is equivalent to the combination of 
\begin{itemize}
    \item $k\leq m$.
    \item If $k>0$ then $\alpha>a(k-1)$.
    \item If $k<m$ then $\alpha<a(k)$.
\end{itemize}
\end{defn}

\begin{lem} \label{k_addable}
Suppose that $1\leq n<\omega$ and $\langle u_b\mid b\in[H]^n\rangle$ is a uniform $n$-dimensional $\Delta$-system as witnessed by $\rho$ and $\langle \mathbf{r_m}\mid \mathbf{m}\subseteq n\rangle$ and assume that $H$ has no largest element. 
If there exists an $\alpha$ which is $k$-addable for $a\in [H]^m$ with $m<n$, set $u_{a,k}=u_b[\mathbf{r}_{m+1\setminus\{k\}}]$ for some $b\in[H]^n$ such that $b[m+1\setminus\{k\}]=a$. 
Let $u_a=u_{a,|a|}$ for each $a\in [H]^m$ for $m<n$.
Then the following hold. 
\begin{enumerate}
    \item These definitions are independent of our choice of $b$
    \item For each $a\in[H]^{m}$ for $m< n$ and $k\leq m$, the collection 
    \[\{u_{a\cup\{\beta\}}\mid\beta\text{ is }k-\text{addable for }a\}\]
    is a $\Delta$-system with root $u_{a,k}$
\end{enumerate}
\end{lem}
\begin{proof}
We first show 1. 
Fix $a\in[H]^{m}$ and suppose that $a=b[m+1\setminus\{k\}]=b'[m+1\setminus\{k\}]$. 
Since $H$ has no largest element, there is a $b''$ such that $b''[m+1]=b[m+1]$ and $b''(j)\not\in b,b'$ for any $j>m+1$. 
Then $b$ and $b''$ are aligned with $\mathbf{r}(b,b'')=m+1$ and $b'$ and $b''$ are aligned with $\mathbf{r}(b',b'')$ either $m+1$ if $b(k)=b'(k)$ or $m+1\setminus \{k\}$ otherwise. 
Regardless, since $\mathbf{r}_{m+1\setminus k}\subseteq\mathbf{r}_{m+1}$, we see that $u_b[\mathbf{r}_{m+1\setminus\{k\}}]=u_{b''}[\mathbf{r}_{m+1\setminus\{k\}}]=u_{b'}[\mathbf{r}_{m+1\setminus\{k\}}]$.

We now show 2. 
If $\alpha\neq\alpha'$ are $k$-addable for $a$, pick $b\sqsupseteq a\cup\{\alpha\}$ and $b'\sqsupseteq a\cup\{\alpha'\}$ such that $b\cap b'=a$. 
Then $b$ and $b'$ are aligned with $\mathbf{r}(b, b')=m+1\setminus\{k\}$, so $u_b$ and $u_{b'}$ are aligned with $\mathrm{r}(u_b,u_b')=\mathbf{r}_{m+1\setminus\{k\}}$. 
Then 
\[\begin{aligned}
u_{a\cup\{\alpha\}}\cap u_{a\cup\{\alpha'\}}&=u_{b}[\mathbf{r}_{m+1}]\cap u_{b'}[\mathbf{r}_{m+1}]\\
&=u_b[\mathbf{r}_{m+1}]\cap u_{b'}[\mathbf{r}_{m+1}]\cap u_b[\mathbf{r}_{m+1\setminus\{k\}}]\cap u_{b'}[\mathbf{r}_{m+1\setminus\{k\}}]\\
&=u_b[\mathbf{r}_{{m+1}\setminus\{k\}}]\cap u_{b'}[\mathbf{r}_{m+1\setminus\{k\}}]\\
&=u_b[\mathbf{r}_{m+1\setminus\{k\}}]\\
&=u_{a,k}
\end{aligned}.\]
\end{proof}

We now cite the main theorem we will use to obtain the existence of higher dimensional $\Delta$-systems; it is a common refinement of both the classical $\Delta$-system lemma and the classical Erd\"os-Rado theorem.

\begin{defn} \label{sigma}
Suppose that $\lambda$ is an infinite regular cardinal. Recursively define cardinals $\sigma(\lambda, n)$ for $1 \leq n<\omega$ by letting $\sigma(\lambda, 1)=\lambda$ and, given $1 \leq n<\omega$, letting $\sigma(\lambda, n+1)=\left(2^{<\sigma(\lambda, n)}\right)^{+}$.
\end{defn}

\begin{fact}[Lambie-Hanson, {\cite[Theorem 2.10]{HDDS}}] \label{DSys_lem}
Suppose that
\begin{itemize}
    \item $1 \leq n<\omega ;$
    \item $\kappa<\lambda$ are infinite cardinals, $\lambda$ is regular and $<\!\!\kappa$-inaccessible, and $\mu=\sigma(\lambda, n)$;
    \item $c:[\mu]^{n} \rightarrow 2^{<\kappa}$;
    \item for all $b \in[\mu]^{n}$, we are given a set $u_{b} \in[\mathrm{On}]^{<\kappa}$.
\end{itemize}

Then there are an $H \in[\mu]^{\lambda}$ and $k<2^{<\kappa}$ such that
\begin{enumerate}
    \item $c(b)=k$ for all $b \in[H]^{n}$;
    \item $\left\langle u_{b} \mid b \in[H]^{n}\right\rangle$ is a uniform $n$-dimensional $\Delta$-system.
\end{enumerate}
\end{fact}

\section{$\Om_{\kappa}$ systems} \label{Om_kap_sec}

Abelian categories are a general place where homological algebra is well behaved; in particular, there is a well behaved notion of exact sequence. 
A canonical example is the category Ab of abelian groups or, more generally, the category of left $R$ modules over a ring $R$.
A question arising in homological algebra is how far a functor between abelian categories is from being \emph{exact}; that is, preserving exact sequences. 
One particular example of this problem comes from the theory of inverse limits. 
On some appropriate abelian category of inverse systems of abelian groups, taking the inverse limit fails to be exact; the canonical example of this phenomenon is the following short exact sequence of systems each indexed by the natural numbers (we call such a system a \emph{tower}):

\begin{center}
\begin{tikzcd}
            \vdots\arrow[d]& \vdots \arrow[d]                           & \vdots \arrow[d]                    & \vdots \arrow[d]                   & \vdots\arrow[d]  \\
0 \arrow[r]\arrow[d] & \mathbb{Z} \arrow[r, "2^n"] \arrow[d, "2"] & \mathbb{Z} \arrow[r] \arrow[d, "1"] & \mathbb{Z}/2^n\ZZ \arrow[r] \arrow[d] & 0\arrow[d] \\
            \vdots\arrow[d]& \vdots \arrow[d, "2"]                      & \vdots \arrow[d, "1"]               & \vdots \arrow[d]                   &\vdots\arrow[d]   \\
0 \arrow[r]\arrow[d] & \mathbb{Z} \arrow[d, "2"] \arrow[r, "4"]   & \mathbb{Z} \arrow[d, "1"] \arrow[r] & \mathbb{Z}/4\ZZ \arrow[d] \arrow[r]   & 0\arrow[d] \\
0 \arrow[r] & \mathbb{Z} \arrow[r, "2"]                  & \mathbb{Z} \arrow[r]                & \mathbb{Z}/2\ZZ \arrow[r]             & 0
\end{tikzcd}
\end{center}
The limit of the second tower is $0$, the limit of the third tower is $\ZZ$, and the limit of the fourth tower is the additive group of 2-adic integers $\ZZ_2$ but the map from $\ZZ$ to $\ZZ_2$ induced by the maps from the third tower to the fourth tower is not onto, so the inverse limit functor $\lim$ does not preserve exact sequences. 
However, $\lim$ does turn out to be \emph{left exact}, meaning that if 
\[0\to{A}\to{B}\to{C}\to0\]
is a short exact sequence of inverse systems then
\[0\to\lim{A}\to\lim{B}\to\lim{C}\]
is an exact sequence of abelian groups. 
Left exactness of $\lim$ follows from the fact that $\lim$ is the right adjoint of the functor which assigns to a group $G$ the system where all groups are $G$ and all maps are the identity; see for instance, \cite[Theorem 2.6.1]{Weibel}.
In a scenario with a left exact functor (starting with a category with the technical assertion that there are \emph{enough injective objects}), the \emph{derived functors} are an attempt to measure the failure of exactness and to recover exactness. 
The derived functors of $\lim$ are a canonical sequence of functors $\lim^s$ which result in a long exact sequence of the form 
\[\begin{tikzcd}
    0\arrow[r]&\lim A\arrow[r]&\lim B\arrow[r]&\lim C\arrow[r]&{}\\
    {}\arrow[r]&\limn{1}A\arrow[r]&\limn{1}B\arrow[r]&\limn{1}C\arrow[r]&{}\\
    {}\arrow[r]&\limn{2}A\arrow[r]&\limn{2}B\arrow[r]&\limn{2}C\arrow[r]&\ldots
\end{tikzcd}\]
for any short exact sequence as above.

Two categories of inverse systems are of particular interest. 
For a fixed poset $\mathbb{P}$, the category $\operatorname{Ab}^\mathbb{P}$ consists of inverse systems of abelian groups indexed over $\mathbb{P}$ and a morphism from $\mathbf{G}$ to $\mathbf{H}$ consists of a family of group homomorphisms from $G_p$ to $H_p$ for each $p\in\mathbb{P}$ commuting with the structure maps of $\mathbf{G}$ and $\mathbf{H}$.
That is, the objects of $\operatorname{Ab}^{\mathbb{P}}$ are functors from $\mathbb{P}^{\operatorname{op}}$ to $\operatorname{Ab}$ with morphisms given by natural transformations. 
The second category of interest is the category pro-Ab, whose objects are inverse systems indexed over directed posets and whose morphisms can be (tersely) given by, for $X$ a system indexed by $\mathbb{P}$ and $Y$ a system indexed by $\mathbb{Q}$,
\[\hom(X,Y)=\operatorname{colim}_{p\in\mathbb{P}}\operatorname{lim}_{q\in\mathbb{Q}}\hom_{\text{Ab}}(X_p,Y_q).\]
The category pro-Ab is abelian and has enough injective objects (see, for instance, \cite[Theorem 15.8 and Theorem 15.15]{SSH}). 
Moreover, the functor $\lim$ is a left exact functor on this category and therefore admits derived functors $\lim^s$ as defined on pro-Ab.
We note that by \cite[Theorem 15.14]{SSH}, if $\mathbb{P}$ is directed then the derived functors of $\lim$ as computed in $\operatorname{Ab}^{\mathbb{P}}$ and as computed in pro-Ab coincide. 

We now describe a generalization of the $\Om$ systems introduced in \cite{ASH} which allows for wider systems. 
An $\Om$ system is a sequence of countably many $\omega$-indexed towers of abelian groups; $\Om_\kappa$ systems allow for $\kappa$ many towers and will allow for statements about derived limits indexed over $\omega^\kappa$.

\begin{defn} \label{Om_kappa sys}
Suppose $\kappa$ is a cardinal. 
An \emph{$\Omega_{\kappa}$ system} $\mathcal{G}$ is specified by an indexed collection $\{G_{\alpha,k} \mid \alpha<\kappa, k \in \omega\}$
of abelian groups along with, for $\alpha<\kappa$ and $j\geq k$, compatible homomorphisms
$p_{\alpha,j,k}: G_{\alpha,j} \to G_{\alpha,k}$. 
Such data give rise to the following additional objects:
\begin{itemize}

\item 
For each $x \in \omega^\kappa$ define 
${\displaystyle G_x := \bigoplus_{\alpha<\kappa} G_{\alpha,x(\alpha)}}$
and
${\displaystyle \overline{G}_x := \prod_{\alpha<\kappa} G_{\alpha,x(\alpha)}}$.
We regard $G_x$ as a subgroup of $\overline{G}_x$.

\item 
For each $x\in\omega^\kappa$ define 
${\displaystyle(\overline{G}/G)_x}=\overline{G}_x/G_x$.

\item
For each $x \leq y \in \omega^\kappa$ a homomorphism
$p_{y,x}:\overline{G}_y \to \overline{G}_x$
defined by $p_{y,x} := \prod_{\alpha<\kappa} p_{\alpha,y(\alpha),x(\alpha)}$; note that $p_{y,x}$ restricts to a homomorphism from $G_y$ to $G_x$ and factors to a map from $(\overline{G}/G)_y$ to $(\overline{G}/G)_x$, both of which we also denote as $p_{y,x}$.

\item 
The systems $G$, $\overline{G}$, and $\overline{G}/G$ indexed over $\omega^\kappa$ with structure given by the above points.

\item 
For each $\alpha<\kappa$ an inverse system $\G_\alpha$ indexed over $\omega$ with $(\G_\alpha)_k=G_{\alpha,k}$ and structure maps given by $p_{\alpha,j,k}$. 
We will often abbreviate $p_{\alpha,k+1,k}$ as $p_{\alpha,k}$. 
We denote the canonical map from $\lim \G_\alpha$ to $G_{\alpha,k}$ as $p_{\alpha,\omega,k}$.

\end{itemize}
\end{defn}

The systems $G,\overline{G},$ and $\overline{G}/G$ are connected through a short exact sequence
\[0\longrightarrow G\longrightarrow\overline{G}\longrightarrow\overline{G}/G\longrightarrow0\]
which gives rise to a long exact sequence of derived limits of the form
\[\begin{aligned}
    0&\longrightarrow &\lim G&\longrightarrow&\lim\overline{G} &\longrightarrow& \lim\overline{G}/G &\longrightarrow\\ & \longrightarrow&\lim{}^1G &\longrightarrow &\lim{}^1\overline{G} & \longrightarrow& \lim{}^1\overline{G}/G&\longrightarrow\\&\longrightarrow&\cdots
\end{aligned}\] 
Ultimately, we will be interested in the derived limits of $G$; we will use information about the derived limits of $\overline{G}$ and $\overline{G}/G$ to compute the derived limits of $G$. 

\begin{exa} \label{A}
Suppose $G_{\alpha,k}=\ZZ^k$ for each $\alpha,k$ and $p_{\alpha,j,k}$ is the canonical projection map. 
The resulting system $G$ is denoted $\mathbf{A}_\kappa$ in, for instance, \cite{SHDLST}.
We may think of $\overline{G}_x$ as consisting of functions from the area under the graph of $x$ to $\ZZ$ and of $G_x$ as the subgroup of finitely supported functions. 
$\overline{G}_x/G_x$ consists of functions from the area under the graph of $x$, where we identify two functions if they agree modulo finite.
\end{exa}
The key ingredient of the proof of Theorem \ref{add} will be to establish connections between the coherence and triviality of certain families with the additivity of derived limits. 
Coherent families will correspond directly to cocycles in a chain complex whose cohomology yields the derived limits of an inverse system. 
Two different versions of triviality will give us different information about how the corresponding element of the derived limit behaves in the above long exact sequence.

\subsection{$\Om_\kappa$ systems and strong homology} \label{kappa_add_sh}

 The main goal of this section is to outline how to generalize the proof of \cite[Theorem 6]{ASH} to show Theorem \ref{add_kappa}. 
We will use the following fact:
\begin{fact} \label{disjoint sum}
    Suppose $X$ is a locally compact metric space of weight at most $\kappa$. 
    There are locally compact separable metric spaces $\langle X_i\mid i<\kappa\rangle$ such that $X\cong\coprod_{i<\kappa}X_i$. 
\end{fact}
This is an easy generalization of {\cite[Theorem 7.3 of chapter XI]{Dug}}.
We now outline the proof of Theorem \ref{add_kappa}.
\begin{proof}
    The proofs of $(\ref{cs})\implies(\ref{add_lc})$ and of $(\ref{add_c})\iff(\ref{add_om})$ can be taken directly from directions $(1)\implies(2)$ and $(2)\iff(3)$ of the proof of \cite[Theorem 6]{ASH}; $(\ref{add_lc})\implies(\ref{add_c})$ is immediate. 
    For the direction $(\ref{add_c})\implies(\ref{cs})$, suppose $X$ is a locally compact metric space of weight at most $\kappa$. 
    By Fact \ref{disjoint sum}, we may fix locally compact separable metric spaces $\langle X_i\mid i<\kappa\rangle$ such that $X\cong\coprod_{i<\kappa}X_i$. 
    For each $i<\kappa$, fix an increasing sequence of open sets $\langle U^i_j\mid j<\omega\rangle$ such that $\overline{U^i_j}$ is compact and $X_i=\bigcup_{j<\omega}U^i_j$.
    Let $U_j=\coprod_{i<\kappa}U^i_j$.
    \begin{claim}
    For each $p<\omega$,
        \[\overline{H}_p(X)\cong \underset{j<\omega}{\colim}\overline{H}_p(\overline{U}_j)\cong\underset{\substack{K\subseteq X\\K\text{ compact}}}{\colim}\overline{H}_p(K).\]
    \end{claim}
    \begin{proof}
        The proof of the first isomorphism is identical to the proof of $(3)\implies(1)$ in the proof of \cite[Theorem 6]{ASH}. 
        For the second isomorphism, by \ref{add_c}, we have $\overline{H}_p(\overline{U_j})\cong \bigoplus_{i<\kappa}\overline{H}_p(\overline{U^i_j})$. 
        By the finite additivity of strong homology, 
        \[\bigoplus_{i<\kappa}\overline{H}_p(\overline{U^i_j})\cong\underset{F\subseteq\kappa\text{ finite}}{\colim}\bigoplus_{i\in F}\overline{H}_p(\overline{U^i_j})\cong \underset{F\subseteq\kappa\text{ finite}}{\colim}\overline{H}_p(\coprod_{i\in F}\overline{U^i_j}).\]
        In particular, since taking two colimits is the same as taking a colimit over the product index set, \[\underset{j<\omega}{\colim}\overline{H}_p(\overline{U_j})\cong\underset{\substack{F\subseteq\kappa\text{ finite}\\j<\omega}}{\colim}\overline{H}_p\left(\coprod_{i\in F}\overline{U^i_j}\right).\] 
        Since finite disjoint unions of the $\overline{U^i_j}$ are cofinal in all compact subsets of $X$, we obtain the second isomorphism.
    \end{proof}
\end{proof}

\subsection{Coherence and Triviality}
We now describe the general machinery by which we will prove additivity of derived limits. 
We will generally make use of the following cochain complex, whose cohomology gives the derived limits of an inverse system indexed over a subset of $\omega^\kappa$; an identical characterization applies to inverse systems indexed over any lower semi-lattice.

\begin{defn} \label{chain_cx}
Suppose that $X\subseteq\omega^\kappa$ and $\mathbf{G}$ is an inverse system of abelian groups indexed over $\omega^\kappa$. 
We say $\Phi\in\prod_{(x_0,...,x_n)\in X^{n+1}}\mathbf{G}_{\bigwedge x_i}$ is \emph{alternating} if for every permutation $\sigma\in S_{n+1}$ and every $\mathbf{x}\in X^{n+1}$, $\Phi_{\mathbf{x}}=\operatorname{sgn}(\sigma)\Phi_{\sigma(\mathbf{x})}$ where $\sigma$ acts by permuting indices. 
The \emph{alternating cochain complex} $K^\bullet(\mathbf{G}\upharpoonright X)$ is given by 

\[K^n(\mathbf{G}\upharpoonright X)=\left\{\Phi\in\prod_{(x_0,...,x_n)\in X^{n+1}}\mathbf{G}_{\bigwedge x_i}\,\middle|\,\Phi\text{ is alternating}\right\}\]

with $d\colon K^n\to K^{n+1}$ given by 
\[d\Phi(\mathbf{x})=\sum_i (-1)^ip_{\bigwedge{\xbf^i},\bigwedge {\xbf}}\Phi(\mathbf{x}^i).\]
In practice, $\mathbf{G}$ can be any of $G,\overline{G}$, or $\overline{G}/G$ from an $\Om_\kappa$ system as in Definition \ref{Om_kappa sys}.
\end{defn}

Note that we do not require $\bigwedge x_i$ to be a member of $X$ in the above definition. 
It is readily verified that $d^{n+1}d^n=0$ for any $n<\omega$; our main object of interest will be the cohomology groups $H^n(K^\bullet(\mathbf{G}\upharpoonright X))$. 

In general, $H^n(K^\bullet(\mathbf{G}\upharpoonright X))$ \emph{does not} correspond to $\lim^n(\mathbf{G}\upharpoonright X)$, the derived limits of the system indexed by those elements of $\omega^\kappa$ which are in $X$. 
For an extreme example, if $X$ consists of two incompatible points $x,y$ then $H^0(K^\bullet(\mathbf{G}\upharpoonright X))$ can be identified with all elements of $G_x\times G_y$ which project to the same element of $G_{x\wedge y}$ whereas $\lim^0(\mathbf{G}\upharpoonright X)$ is all of $G_x\times G_y$. 
Nevertheless, $H^n(K^\bullet(\mathbf{G}\upharpoonright X))$ does have a characterization in terms of derived limits. 
We first cite a theorem due to Leray (see, for instance, \cite[Exercise III.4.11]{HARTSHORNE}): 
\begin{fact} \label{Leray}
    Let $Y$ be a topological space, $\mathscr{F}$ a sheaf on $Y$ and $\mathscr{U}=\left\{U_i\right\}_{i \in I}$ an open cover of $Y$. If the sheaf cohomology group $\mathrm{H}^p(V, \mathscr{F})=0$ for all $p>0$ and all $V=U_{i_0} \cap \ldots U_{i_n}, i_j \in I$ then the \v{C}ech cohomology group $\check{\mathrm{H}}^p(\mathscr{U}, \mathscr{F}) $ coincides with the sheaf cohomology group $ \mathrm{H}^p(Y, \mathscr{F})$ for all $p$.
\end{fact}
The characterization of $H^n(K^\bullet(\mathbf{G}\upharpoonright X))$ is the following:

\begin{prop} \label{lim char}
$H^n(K^\bullet(\mathbf{G}\upharpoonright X))\cong \lim^n(\mathbf{G}\upharpoonright Y)$ where $Y$ is the downwards closure of $X$ with respect to the coordinatewise domination ordering. 
\end{prop}

\begin{proof}
An observation in \cite[page 4]{JENSEN} shows that $\lim^n(\mathbf{G}\upharpoonright Y)$ is naturally isomorphic to the sheaf cohomology $H^n(Y,\mathcal{F})$, where the topology on $Y$ is generated by the sets $U_x=\{y\mid y\leq x\}$ for $x\in Y$ and $\mathcal{F}$ is the sheaf $U\mapsto \lim(\mathbf{G}\upharpoonright U)$. 
We now consider the cover $\mathcal{U}=\{U_x\mid x\in X\}$; observe that these sets cover $Y$ and that $U_{x_1}\cap\ldots\cap U_{x_n}=U_{x_1\wedge\ldots\wedge x_n}$. 
Since $U_{x_1\wedge\ldots\wedge x_n}$ has a maximal element, $\lim^s(\mathbf{G}\upharpoonright U_{x_1\wedge\ldots\wedge x_n})=0$ for all $s>0$.
By Fact \ref{Leray}, the $n$th \v Cech cohomology of $\mathcal{U}$ is $\lim^n(\mathbf{G}\upharpoonright Y)$.
Since \v Cech cohomology can be computed using alternating chains by \cite[I.20 Proposition 2]{SERRE}, the proposition follows. 
\end{proof}

\begin{rem}
    The downwards closure of $X$ in Proposition \ref{lim char} can be replaced with any $Y$ satisfying
    \begin{itemize}
        \item $X\subseteq Y$.
        \item all finite meets of elements of $X$ are in $Y$.
        \item every element of $Y$ is below some element of $X$.
    \end{itemize}
\end{rem}

Later in this paper, the $X$ we will be interested in will consist of many mutually generic Cohen sequences; in particular, elements of $X$ will be pairwise incomparable with respect to coordinatewise domination and $X$ will not be closed under finite meets.

From this point on, we will sometimes abuse notation and use $\lim^n(\mathbf{G}\upharpoonright X)$ to refer to $\lim^n(\mathbf{G}\upharpoonright Y)$ for $Y$ the downwards closure of $X$. 
When we write $K^\bullet(\mathbf{G}\upharpoonright X)$, we will always mean that indices are to be taken exclusively from $X$.

The short exact sequence of inverse systems
\[\begin{tikzcd}
    0\arrow[r]&G\arrow[r]&\overline{G}\arrow[r]&\overline{G}/G\arrow[r]&0
\end{tikzcd}\]
naturally gives rise to a short exact sequence of cochain complexes 
\[\begin{tikzcd}
    0\arrow[r]&K^\bullet(G\upharpoonright X)\arrow[r]&K^\bullet(\overline{G}\upharpoonright X)\arrow[r]&K^\bullet(\overline{G}/G\upharpoonright X)\arrow[r]&0
\end{tikzcd}\]
and, by the snake lemma, the long exact sequence of derived limits

\[\begin{aligned}
    0&\longrightarrow \lim G\upharpoonright X&\longrightarrow &\lim\overline{G}\upharpoonright X &\longrightarrow &\lim\overline{G}/G\upharpoonright X \longrightarrow\\ & \longrightarrow\lim{}^1G\upharpoonright X &\longrightarrow &\lim{}^1\overline{G}\upharpoonright X & \longrightarrow &\lim{}^1\overline{G}/G\upharpoonright X\longrightarrow\\&\longrightarrow\cdots
\end{aligned}.\]

\begin{defn} \label{coherent}
Fix $X\subseteq\omega^\kappa$ and a positive integer $n$. 
$\Phi\in K^{n-1}(\overline{G}\upharpoonright X)$ is \emph{$n$-coherent for $\mathcal{G}$} if $d\Phi\in K^{n}(G\upharpoonright X)$. 
Equivalently, the image of $\Phi$ in $K^{n-1}(\overline{G}/G\upharpoonright X)$ is a cocycle. 
\end{defn}
Coherent families represent cocycles in $K^{n}(\overline{G}/G\upharpoonright X)$ and in turn represent elements of $\lim^n(\overline{G}/G\upharpoonright X)$. 
For an $n$-coherent family $\Phi$, we write $[\Phi]$ to mean the element of $\lim^n(\overline{G}/G\upharpoonright X)$ represented by $\Phi$. 

\begin{exa}
Suppose $\mathcal{G}$ is as in Example \ref{A} and $X\subseteq\omega^\kappa$. 
An $n$-coherent family corresponds to an alternating family of functions $\{f_{\mathbf{x}}\mid \mathbf{x}\in X^n\}$ from the area under the graph of $\bigwedge\mathbf{x}$ into $\ZZ$ such that for $\mathbf{x}\in X^{n+1}$, $\sum_{i}(-1)^if_{\mathbf{x}^i}$ is finitely supported below the graph of $\bigwedge\mathbf{x}$. 
Note that a restriction to the area under the graph of $\bigwedge\mathbf{x}$ is required for the domains of these functions to align.
This is the notion of $n$-coherent families used for instance in \cite{SHDLST}. 
\end{exa}

We now give two notions of triviality; they both will provide information about how a coherent family interacts with the above long exact sequence. 
In general, the two notions of triviality differ, but on the class of injective systems we will see that they do not. 

\begin{defn}\label{I-trivial}
Fix $X\subseteq\omega^\kappa$ and a positive integer $n$. 
If $n=1$, a $1$-coherent family $\Phi\in K^0(\overline{G}\upharpoonright X)$ is \emph{type $I$ trivial} if there exists a $\psi\in\lim\overline{G}$ such that for each $f\in X$ there is a finite $S$ such that $\Phi_f\upharpoonright(\kappa\setminus S)=(p_f\psi)\upharpoonright(\kappa\setminus S)$, where $p_f$ is the canonical projection map from $\lim\overline{G}$ to $\overline{G}_f$. 
In this case, we say that $\Phi_f=^*\psi_f$. 

If $n>1$, an $n$-coherent family $\Phi\in K^{n-1}(\overline{G}\upharpoonright X)$ is \emph{type $I$ trivial} if there exists a $\Psi\in K^{n-2}(\overline{G}\upharpoonright X)$ such that for each $\mathbf{x}\in X^n$ there is a finite $S$ such that $\Phi_{\mathbf{x}}\upharpoonright\prod_{i\in \kappa\setminus S} G_{i,\bigwedge\mathbf{x}(i)}=d\Psi_{\mathbf{x}}\upharpoonright\prod_{i\in \kappa\setminus S} G_{i,\bigwedge\mathbf{x}(i)}$. 
Equivalently, $[d\Psi]=[\Phi]$ in $K^{n-1}(\overline{G}/G\upharpoonright X)$.
In this case, we say that $\Phi_{\mathbf{x}}=^*d\Psi_{\mathbf{x}}$. 

Such a $\psi$ or $\Psi$ is referred to a \emph{type $I$ trivialization}.
\end{defn}

\begin{exa}
    Suppose $\mathcal{G}$ is as in Example \ref{A}. 
    Then $\lim\overline{G}$ consists of all functions from $\kappa\times\omega$ to $\ZZ$ so when $n=1$, type $I$ triviality is equivalent to the existence of an $f\colon \kappa\times\omega\to\ZZ$ such that for each $x$, $\Phi_x=^*f\upharpoonright I_x$ for $I_x$ the area under the graph of $x$.
\end{exa}

Note that when $n>1$, $\Phi$ is type $I$ trivial if and only if $[\Phi]=0$ in \[\lim{}^{n-1}(\overline{G}/G\upharpoonright X)\] and when $n=1$, $\Phi$ being type $I$ trivial is equivalent to $[\Phi]$ being in the image of the map 
\[\lim{}^{0}(\overline{G}\upharpoonright X)\to \lim{}^0(\overline{G}/G\upharpoonright X).\]
By subtracting $\Phi$ and the projections of $\psi$ in the $n=1$ case, we see that the $n=1$ case is also equivalent to the existence of a $\Psi=\langle \Psi_f\,\mid\,f\in X\rangle$ such that $d\Psi=0$ and $\Psi_{f}=^*\Phi_{f}$. 

For general $\Om_{\kappa}$ systems, we cannot hope for every $n$-coherent family to have a type $I$ trivialization; however, type $I$ triviality is the sort of triviality which we can propagate from small subsets of $\omega^\kappa$ to larger subsets of $\omega^\kappa$.
Another sort of triviality which can hold, and was used in \cite{ASH}, is the following.

\begin{defn} \label{II-trivial}
Fix $X\subseteq\omega^\kappa$ and a positive integer $n$. 
An $n$-coherent family $\Phi\in K^{n-1}(\overline{G}\upharpoonright X)$ is \emph{type $II$ trivial} if there exists a $\Psi\in K^{n-1}(G\upharpoonright X)$ such that $d\Phi=d\Psi$.
Such a $\Psi$ is referred to as a \emph{type $II$ trivialization}.
\end{defn}

Chasing through the proof of the snake lemma, $\Phi$ being type $II$ trivial is equivalent to the image of $\Phi$ under the connecting homomorphism from $\lim^{n-1}\overline{G}/G$ to $\lim^nG\upharpoonright X$ being $0$. 
Note that $\Phi$ being type $I$ trivial implies that $\Phi$ is type $II$ trivial: if $n>1$ and $\Psi$ is a type $I$ trivialization then $\Phi-d\Psi$ is a type $II$ trivialization if $n=1$ and $\Psi$ is a type $I$ trivialization then $\Phi-\Psi$ is a type $II$ trivialization. However, for general $\Om_\kappa$ systems there can be coherent families which are type $II$ trivial but not type $I$ trivial as the following example shows:
\begin{exa}
    Consider the $\Omega_\omega$ system with each $G_{n,k}=\ZZ$ and each $p_{n,k+1,k}$ given by multiplication by $2$. We saw earlier a short exact sequence of towers of this form where $\lim$ failed to be exact given by \begin{center}
\begin{tikzcd}
            \vdots\arrow[d]& \vdots \arrow[d]                           & \vdots \arrow[d]                    & \vdots \arrow[d]                   & \vdots\arrow[d]  \\
0 \arrow[r]\arrow[d] & \mathbb{Z} \arrow[r, "2^n"] \arrow[d, "2"] & \mathbb{Z} \arrow[r] \arrow[d, "1"] & \mathbb{Z}/2^n\ZZ \arrow[r] \arrow[d] & 0\arrow[d] \\
            \vdots\arrow[d]& \vdots \arrow[d, "2"]                      & \vdots \arrow[d, "1"]               & \vdots \arrow[d]                   &\vdots\arrow[d]   \\
0 \arrow[r]\arrow[d] & \mathbb{Z} \arrow[d, "2"] \arrow[r, "4"]   & \mathbb{Z} \arrow[d, "1"] \arrow[r] & \mathbb{Z}/4\ZZ \arrow[d] \arrow[r]   & 0\arrow[d] \\
0 \arrow[r] & \mathbb{Z} \arrow[r, "2"]                  & \mathbb{Z} \arrow[r]                & \mathbb{Z}/2\ZZ \arrow[r]             & 0
\end{tikzcd}
\end{center}
All higher derived limits of a constant system vanish (for instance as an easy application of \cite[Theorem 14.9]{SSH}) so exactness implies that the first derived limit of the leftmost tower should be $\mathbb{Z}_2/\mathbb{Z}$ for $\mathbb{Z}_2$ the $2$-adic integers. 
Chasing through the Snake lemma to produce a representative of a nonzero element of $\lim^1$ and copying it in all columns yields the following example of a coherent family. 

    Fix a sequence of integers $\langle a_i\mid i<\omega\rangle$ such that 
    \begin{itemize}
        \item if $i<j$ then $a_i\equiv a_j\mod 2^i$;
        \item there is no $b\in\mathbb{Z}$ such that $b\equiv a_i\mod 2^i$ for each $i$. 
    \end{itemize}
    For $i\leq j$, let $b_{ij}=\frac{a_j-a_i}{2^i}$ and $b_{ji}=-b_{ij}$. 
    Let $\Phi$ be the family given by $\Phi_{xy}(i)=b_{x(i)y(i)}$; an easy computation shows that $d\Phi=0$. 
    In particular, $\Phi$ is 2-coherent and type $II$ trivial with a trivialization given by the constant $0$ system. 
    We show that $\Phi$ is not type $I$ trivial.

    Suppose otherwise and let $\Psi$ be a type $I$ trivialization.
    Let $X\subseteq\omega^\omega$ be $\leq^k$ cofinal for some $k$ such that whenever $x\in X$ and $i\geq k$, $2^{x(i)}\Psi_x(i)-\Psi_{\underline{0}}(i)=b_{0,x(i)}$ for $\underline{0}$ the constant $0$ function. 
    Let $A=\{x(k)\mid x\in X\}$ and pick for each $i\in A$ some $x_i\in X$ with $x_i(k)=i$; note that $A$ is infinite. 
    Moreover, whenever $i\leq j$ are in $A$, 
    \[\begin{aligned}
        2^{j-i}\Psi_{x_j}(k)-\Psi_{x_i}(k)&=\frac{1}{2^i}\left(2^j\Psi_{x_j}(k)-2^i\Psi_{x_i}(k)\right)\\
        &=\frac{1}{2^i}\left(2^j\Psi_{x_j}(k)-\Psi_{\underline{0}}(k)-2^i\Psi_{x_i}(k)+\Psi_{\underline{0}}(k)\right)\\
        &=\frac{1}{2^i}\left(a_j-a_0-a_i+a_0\right)\\
        &=b_{ij}.
    \end{aligned}\]

    Since
    $b_{ij}$ represents a nonzero element of $\lim^1$ for the tower with all groups $\ZZ$ and all maps given by multiplication by $2$, $b_{ij}$ is not a coboundary along any cofinal subset of $\omega$. That is, whenever $B\subseteq\omega$ is infinite and $\langle c_i\mid i\in B\rangle$ are integers, there are $i<j$ such that $2^{j-i}c_j-c_i\neq b_{ij}$. 
    But setting $B$ to be the $A$ of the preceding paragraph and setting $c_i=\Psi_{x_i}(k)$ yields the desired contradiction. 
\end{exa}
When $n=1$, however, if $\Psi$ is a type $II$ trivialization of $\Phi$ then we may glue together the group elements appearing in $\Phi-\Psi$ to obtain a type $I$ trivialization, leading to the following fact.
\begin{fact} \label{trivequiv}
    If $n=1$ and $\Phi$ is $n$-coherent then $\Phi$ is type $I$ trivial if and only if $\Phi$ is type $II$ trivial. 
\end{fact} 

More generally, the only problem arises when $n=2$; we will not need this remark and simply state it without proof. 
The special case for $\Omega_\omega$ systems follows easily from \cite[Lemma 5]{ASH} and the more general case follows the same proof. 
\begin{rem}
    If $n\neq 2$ and $\Phi$ is $n$-coherent then $\Phi$ is type $I$ trivial if and only if $\Phi$ is type $II$ trivial.
\end{rem}

Our overall plan, as was done in \cite{SVHDLwLC} for the system $\mathbf{A}$, is to obtain a type $II$ trivialization over a large set and then propagate the trivialization to obtain that the entire system is type $I$ trivial. 
In general, there can be coherent families with type $II$ trivializations but no type $I$ trivializations.
However, there is a rich class of systems, namely the injective systems, where the questions of whether a coherent family has a type $I$ trivialization and whether the same coherent family has a type $II$ trivialization are equivalent (see Lemma \ref{equiv}).

\subsection{Injective Systems}
Several features of the system $\mathbf{A}$ are used heavily in \cite{SVHDLwLC}, but do not apply to arbitrary $\Om_\kappa$ systems, even when $\kappa=\omega$. 
In this section, we analyze the injective objects among the $\Om_\kappa$ systems, which do share these features with the system $\mathbf{A}$. 
Of particular interest are

\begin{itemize}
\item for injective $\Om_\kappa$ systems, the two notions of when a coherent family is trivial are equivalent (see Lemma \ref{equiv}).

    \item additivity of derived limits for injective systems is equivalent to vanishing of derived limits, which is again equivalent to every coherent family being trivial (see Corollary \ref{lim_vanish}).
    
    \item there is a well-behaved notion of ``extension by 0,'' allowing us to extend the analysis of coherent families to coherent families below some fixed $x$ (see Lemma \ref{concrete inj} and Proposition \ref{trivial_below}).
\end{itemize}

The collection of $\Om_\kappa$ systems forms an abelian category which we denote by $\Om_\kappa$ sys, with a morphism $\varphi$ between two systems $\mathcal{G}$ and $\mathcal{H}$ consisting of group homomorphisms $\varphi_{\alpha,k}\colon G_{\alpha,k}\to H_{\alpha,k}$ commuting with the structure maps of $\mathcal{G}$ and $\mathcal{H}$. 
Moreover, each $\varphi_{\alpha,k}$ is injective if and only if $\varphi$ is monic in the sense that if $\mathcal{K}$ is any other $\Om_\kappa$ system and $\psi_1,\psi_2\colon\mathcal{K}\to \mathcal{G}$ satisfy $\varphi\circ\psi_1=\varphi\circ\psi_2$ then $\psi_1=\psi_2$ (see, for instance, \cite[Theorem 11.2]{SSH}). 
In any abelian category, we may identify the class of injective objects; in the class of $\Omega_\kappa$ systems injective objects are as follows: 

\begin{rem}
An $\Om_\kappa$ system $\mathcal{I}$ is injective if whenever $\mathcal{G},\mathcal{H}$ are $\Om_\kappa$ systems and $\varphi\colon\mathcal{G}\to\mathcal{I}$ and $\psi\colon\mathcal{G}\to\mathcal{H}$ are maps of $\Om_\kappa$ systems such that each $\psi_{\alpha,k}$ is injective, there is a $\overline{\varphi}\colon\mathcal{H}\to\mathcal{I}$ such that $\overline{\varphi}\psi=\varphi$. 
\end{rem}

Note that there is a sensible notion of injective objects for $\mathrm{Ab}^\mathbb{P}$ for any poset $\mathbb{P}$, which we will make use of in Lemma \ref{propagate}. 
See, for instance, \cite[Chapter 11]{SSH} for more details on the general theory of such objects.

Injective objects play an important role in the theory of derived functors; see, for instance, \cite[Section 2.3]{Weibel}. 
The injective objects in abelian groups consist precisely of the \emph{divisible} abelian groups; that is, the groups where for each group element $x$ and each $n\in\ZZ\setminus\{0\}$ there is a $y$ such that $ny=x$.

We will prove the class of injective towers of abelian groups admits the following concrete characterization; note that an $\Om_\kappa$ system is injective if and only if each of its towers is.

\begin{lem} \label{concrete inj}
    Suppose $\mathbf{G}=\langle G_i,p_{i}\mid i<\omega\rangle$ is a tower of abelian groups. 
    $\mathbf{G}$ is injective if and only if 
    \begin{itemize}
        \item $G_0$ is a divisible abelian group.
        \item for each $i<\omega$, $p_i$ is surjective.
        \item for each $i<\omega$, $\ker(p_i)$ is a divisible abelian group.
    \end{itemize}
\end{lem}
\begin{proof}
    We first show these conditions are sufficient for injectivity. 
    Fix $\mathbf{H},\mathbf{H}'$ towers of abelian groups, an $f\colon \mathbf{H}\to\mathbf{H}'$ with each $f_{\alpha,k}$ injective, and a $g\colon\mathbf{H}\to\mathbf{G}$. 
    Consider the poset of pairs $(\mathbf{G}^*,g^*)$ where $\mathbf{G}^*$ is a subsystem of $\mathbf{H}'$ containing the image of $\mathbf{H}$ under $f$ and the following diagram commutes:
    \[
    \begin{tikzcd}
        H\arrow{rr}{g}\arrow{dr}{f}&&G\\
        &G^*\arrow{ur}[swap]{g^*}&
    \end{tikzcd}\] 
    ordered in the obvious way. 
    This poset is readily verified to satisfy the hypotheses of Zorn's lemma, so let $(\mathbf{M},h)$ be a maximal element; say $\mathbf{M}=\langle M_i,m_i\mid i<\omega\rangle$. 
    We claim that $\mathbf{M}=\mathbf{H}'$. 
    Suppose not and let $n$ be minimal such that $M_n\neq H_n'$. 
    There are two cases. 
    \begin{itemize}
        \item Case 1: there is some $0<k<\omega$ and $x\in H_n'\setminus M_n$ such that $kx\in M_n$. 
        Then fix such an $x,k$ for the least such value of $k$. 
        If $n>0$ then using the surjectivity of $p_{n-1}$ fix a $y$ such that $p_{n-1}(y)=h_{n-1}(m_{n-1}(x))$; note that $m_{n-1}(x)\in\dom(h_{n-1})$ by minimality of $n$. 
        Then $ky-h_n(kx)\in\ker(p_{n-1})$ so by hypothesis there is a $z\in\ker(p_n)$ such that $kz=ky-h_n(kx)$. 
        We now have the following diagram:
        \[\begin{tikzcd}
            M_n\oplus\ZZ\arrow{rr}{h_n\oplus(y-z)}\arrow{dr}{id\oplus x}&&G_n\\
            &M_n+\ZZ x&
        \end{tikzcd}.\]
        Since $id\oplus x$ is a surjective group homomorphism, $id\oplus x$ is a quotient map and since $h\oplus(y-z)$ vanishes on $\ker(id\oplus x)=\{(k\ell x,-k\ell)\mid\ell\in \ZZ\}$, there exists a unique group homomorphism $h_n^+\colon M_n+\ZZ x\to G$ making the diagram
        \[\begin{tikzcd}
            M_n\oplus\ZZ\arrow{rr}{h_n\oplus(y-z)}\arrow{dr}{id\oplus x}&&G_n\\
            &M_n+\ZZ\arrow{ur}{h_n^+}x&
        \end{tikzcd}.\]
        commute. 
        Moreover, whenever $u\in M_n$ and $\ell\in\ZZ$, 
        \[\begin{aligned}
            h_{n-1}m_{n-1}(u+\ell x)&=h_{n-1}m_{n-1}(u)+h_{m-1}m_{n-1}(\ell x)\\
            &=p_{n-1}h_n(u)+p_{n-1}(\ell y)\\
            &=p_{n-1}h_n^+(u)+p_{n-1}(\ell(y-z))\\
            &=p_{n-1}h_n^+(u)+p_{n-1}h_{n}^+(\ell x)\\
            &=p_{n-1}h_n^+(u+\ell x).
        \end{aligned}\]
        where the first equality is that $h_{n-1}$ and $m_{n-1}$ are homomorphisms, the second equality is that the maps $m_i$ already formed a map of systems and the choice of $y$, the third equality is that $h_n^+$ agrees with $h_n$ on $M_n$ and that $z\in\ker(p_{n-1})$, the fourth is the definition of $h_n^+(x)$, and the last equality is that $p_{n-1}$ and $h_n^+$ are group homomorphisms. 
        In particular, the system $\mathbf{M}^+$ with 
        \[\mathbf{M}^+_k=\left\{\begin{array}{cc}
             M_k&k\neq n  \\
             M_n+\ZZ x&k=n 
        \end{array}\right.\]
        and maps $h_k$ for $k\neq n$ and $h_n^+$ for $k=n$ is in the poset and strictly greater than $(\mathbf{M},h)$ contradicting the maximality of $(\mathbf{M},h)$.
        
        If $n=0$ in this case, set $y=0$ and find any $z$ satisfying $kz=h_0(kx)$ using divisibility of $\mathbf{G}_0$.
        Then identical reasoning shows that we may extend to a strictly larger system by declaring $h_n^+(x)=z$. 

        \item Case 2: otherwise. 
        Then fix any $x\not\in\dom(h_n)$; then $h$ can be extended to the subsystem generated by $M,x$ by setting $h_n(x)=y$ for any $y$ such that if $n>0$ then $p_{n-1}(y)=h_{n-1}m_{n-1}(x)$; such a $y$ exists since $p_{n-1}$ is surjective (note that if $n=0$ we place no requirements on $y$). 
        This is well defined since the canonical map $M_n\oplus\ZZ\to M_n+x$ is an isomorphism in this case. 
        
    \end{itemize}
    This completes the proof that these conditions are sufficient for injectivity. 

    To see that these conditions are necessary, we show each bullet point. 
    If some $x\in G_0$ is not divisible by some $0<k<\omega$, let $\mathbf{H}$ be the system with 
    \[H_i=\left\{\begin{array}{cc}
         (G_0\oplus\ZZ)/\langle x,-k\rangle&i=0  \\
         G_i&\text{otherwise} 
    \end{array}\right.\]
    with maps induced by the $p_i$. 
    Then the identity from $\mathbf{G}$ to $\mathbf{G}$ does not extend to a map from $\mathbf{H}$ to $\mathbf{G}$ as the image of $(0,1)$ in any extension $i$ would satisfy $ki(0,1)=x$. 

    For the second bullet point, suppose $x\in G_i$ is not in the image of $p_i$. 
    Let $\mathbf{H}$ be the system with 
    \[H_n=\left\{\begin{array}{cc}
         G_n\oplus\ZZ&n=i+1  \\
         G_n&\text{otherwise}
    \end{array}\right.\]
    and maps given by $p_n$ if $n\neq i$ and if $n=i$ then the map restricting to $p_i$ on $G_{i+1}$ and mapping $1\in\ZZ$ to $x$. 
    Then the identity map on $\mathbf{G}$ does not extend to a map from $\mathbf{H}$ to $\mathbf{G}$: if $j$ were such an extension, we would need $x=p_ij_{i+1}(0,1)$ contrary to choice of $x$. 

    Finally, suppose for some $x\in\ker(p_i)$ there is a $k$ such that no $y\in\ker(p_i)$ satisfies $ky=x$. 
    Let $\mathbf{H}$ be the system with 
    \[H_n=\left\{\begin{array}{cc}
         (G_n\oplus\ZZ)/\langle (x,-k)\rangle&n=i+1  \\
         G_n&\text{otherwise} 
    \end{array}\right.\]
    with maps given by $p_n$ if $n\neq i$ and induced from the map $(z,\ell)\mapsto p_{i}(z)$ if $x=i-1$. 
    Then the identity map on $\mathbf{G}$ does not extend to a map from $\mathbf{H}$ to $\mathbf{G}$: if $j$ were any such extension, we would need $j(0,1)\in\ker(p_i)$ and $kj(0,1)=x$ contrary to choice of $x$.
\end{proof}

We are led to a particularly nice structure theorem for injective towers.
\begin{cor} \label{inj iso}
    A tower of abelian groups $\mathbf{G}=\langle G_i,p_i\mid i<\omega\rangle$ is injective if and only if there are divisible abelian groups $\langle H_i\mid i<\omega\rangle$ and isomorphisms $f_i\colon G_i\to \prod_{j\leq i}H_i$ such that whenever $i<j$, the following diagram commutes for the right hand map the natural projection map:
    \[
\begin{tikzcd}
G_i \arrow[r, "f_i"] \arrow[d, "p_{ij}"] & \prod_{k\leq i}H_k \arrow[d] \\
G_j \arrow[r, "f_j"]                     & \prod_{k\leq j}H_k             
\end{tikzcd}
\]
\end{cor}
\begin{proof}
    The existence of such $H_i$, $f_i$ trivially ensures the conditions of Lemma \ref{concrete inj}. 
    Conversely, suppose $\mathbf{G}$ is injective. 
    Set $H_i=\ker(p_i)$; we construct $f_i$ recursively on $i$.
    When $i=0$, $G_0=H_0$ and we let $f_0$ be the identity map. 
    Now suppose $i>0$.
    Since $\ker(p_i)$ is a divisible abelian group and $p_i$ is surjective by Lemma \ref{concrete inj}, there is a split exact sequence
    \[0\longrightarrow \ker(p_i)\longrightarrow G_{i+1}\longrightarrow G_i\longrightarrow 0.\]. 
    In particular, there is the desired isomorphism $f_{i+1}\colon G_{i+1}\cong G_i\oplus \ker(p_i)$ making the diagram
    \[
\begin{tikzcd}
G_{i+1} \arrow[rr, "f_{i+1}"] \arrow[rd, "p_i"] &     & G_i\oplus\ker(p_i) \arrow[ld] \\
                                                & G_i &                              
\end{tikzcd}\]
commute (see \cite[Exercise 1.4.2]{Weibel}).
\end{proof}

We now record some facts about injective objects within $\Om_{\kappa}$ systems. 
The following is a standard fact about categories of the form $\mathrm{Ab}^\mathbb{P}$ for any poset $\mathbb{P}$.

\begin{lem}
The category of $\Om_\kappa$ systems has enough injective objects. 
That is, whenever $\mathcal{G}$ is an $\Om_\kappa$ system there is an injective system $\mathcal{I}$ and a monic $\varphi\colon\mathcal{G}\to\mathcal{I}$.
\end{lem}
See, for instance, \cite[Theorem 11.18]{SSH} for the general proof. 

Recall that to any $\Om_\kappa$ system $\mathcal{G}$ we may associate two inverse systems $G$ and $\overline{G}$ indexed over $\omega^\kappa$; for injective systems $\mathcal{I}$ we denote these $I$ and $\overline{I}$.
One feature of the system $\mathbf{A}$ of Example \ref{A} used prominently in \cite[Proposition 4.2]{SVHDLwLC} is that there are compatible maps upwards from $\overline{G}_f$ to $\overline{G}_g$ for each $f\leq g$, namely by extending a function from the area under the graph of $f$ to $\ZZ$ to a function from the area under the graph of $g$ to $\ZZ$ by writing $0$ in all undefined positions. 
The next lemma will allow us to emulate the procedure of extension by $0$ for injective systems.

\begin{cor} \label{extend_by_0}
Suppose $\mathcal{I}$ is an injective $\Om_\kappa$ system. 
There is a family $\{i_{fg}\colon\overline{I}_f\to\overline{I}_g\mid f\leq g\}$ of group homomorphisms such that
\begin{itemize}
    \item $p_{gf}i_{fg}=\mathrm{id}_{\overline{I_f}}$;
    \item if $x\in I_f$ then $i_{fg}(x)\in I_g$.
\end{itemize}

\end{cor}

\begin{proof}
By Corollary \ref{inj iso}, we may fix groups $\langle H_{\alpha,i}\mid \alpha<\kappa,i<\omega\rangle$ and isomorphisms $h_{\alpha,i}\colon I_{\alpha,i}\cong\prod_{j\leq i}H_j$ commuting with the natural projection maps. 
For $i<j$ and $\alpha<\kappa$, let $\ell_{\alpha,i,j}\colon \prod_{k\leq }H_{\alpha,i}\to \prod_{k\leq j}H_{\alpha,j}$ be the map given by extending by $0$. 
Then setting \[i_{fg}=\prod_{\alpha<\kappa}h_{\alpha,g(\alpha)}^{-1}\circ \ell_{\alpha,f(\alpha),g(\alpha)}\circ h_{\alpha,f(\alpha)}\]
will do.
\end{proof}

We now turn to establishing that the two forms of triviality are equivalent for injective systems. 
This will allow us to emulate the ``propagating trivializations'' step of \cite[Section 4]{SVHDLwLC}. 

\begin{lem} \label{pullback}
Suppose that $I\in\Ab^\omega$ is injective and $X\subseteq\omega^\kappa$. 
Then for each $\alpha<\kappa$ the system $I(\alpha)$ indexed over $X$ defined by $I(\alpha)_x=I_{x(\alpha)}$ is injective in $\Ab^X$.
\end{lem}

\begin{proof}
Let $\pi_\alpha^*\colon\Ab^\omega\to\Ab^X$ be the functor with $\pi_\alpha^*(\mathbf{G})(x)=G_{x(\alpha)}$; then $\pi_\alpha^*(I)=I(\alpha)$. 
The key point is that
$\pi_\alpha^*$ has a left adjoint $(\pi_\alpha)_!$ which is exact; see \cite[Lemma 14.11]{SSH}.
Therefore, since both $\Ab^\omega$ and $\Ab^X$ are abelian categories, $\pi_\alpha^*$ preserves injective objects by \cite[Proposition 2.3.10]{Weibel}.
\end{proof}

Recall that we use $\limn{n}(\overline{I}\upharpoonright X)$ as a slight abuse of notation for the derived limits of the system indexed by the downwards closure of $X$, which is by Proposition \ref{lim char} the cohomology group $H^{n}(K^\bullet(\overline{I}\upharpoonright X))$.
\begin{cor} \label{lim_vanish}
Suppose that $\mathcal{I}=\langle I_{\alpha,k},p_{\alpha,j,k}\rangle$ is an injective object in the category of $\Omega_{\kappa}$ systems. 
Then for all $X\subseteq\omega^\kappa$ and $n>0$
\[\lim{}^n(\overline{I}\upharpoonright X)=0.\]
\end{cor}

\begin{proof}
Let $Y$ be the downwards closure of $X$.
Note that for each $\alpha$, $\mathcal{I}_\alpha$ is injective in $\Ab^\omega$. 
Therefore, each $\pi_\alpha^*(\mathcal{I}_\alpha)$ is injective in $\Ab^Y$ by Lemma \ref{pullback}. 
In particular, noting that $\overline{I}\upharpoonright Y=\prod_\alpha\pi_\alpha^*(\mathcal{I}_\alpha)$ where the product is taken in $\mathrm{Ab}^Y$ and that a product of injective objects is injective (in any category), $\overline{I}\upharpoonright Y$ is an injective object, so derived functors for any left exact functor (for instance $\lim$) will vanish.
\end{proof}

\begin{rem} \label{non_lim_vanish}
    The conclusion of Corollary \ref{lim_vanish} can fail in the case $n=1$ for general $\Om_\kappa$ systems. 
    For instance, suppose $\mathcal{G}$ is the $\Om_\kappa$ system with all groups $\ZZ$ and all $p_{\alpha,n+1,n}$ being multiplication by 2. 
    Then $\lim^1\overline{G}\cong\prod_{\alpha<\kappa}\ZZ_2/\ZZ$ for $\ZZ_2$ the additive group of 2-adic integers. 
\end{rem}

The next lemma shows that two types of triviality are equivalent for injective systems. 
\begin{lem} \label{equiv}
Suppose that $\mathcal{I}$ is an injective $\Om_{\kappa}$ system. 
For every $X\subseteq\omega^\kappa$, positive integer $n$, and every $n$-coherent $\Phi\in K^{n-1}(\overline{I}\upharpoonright X)$, $\Phi$ is type $I$ trivial if and only if $\Phi$ is type $II$ trivial. 
\end{lem}
\begin{proof}
By Corollary \ref{lim_vanish}, $\lim^n\overline{I}\upharpoonright X=0$ for every $n>0$. 
Letting $Y$ be the downwards closure of $X$ and using the long exact sequence of derived limits corresponding to 
\begin{center}
\begin{tikzcd}
0 \arrow[r] & {I}\upharpoonright Y \arrow[r] & \overline{{I}}\upharpoonright Y \arrow[r] & \overline{I}/I\upharpoonright Y \arrow[r] & 0,
\end{tikzcd}
\end{center}
there is an exact sequence
\begin{center}
    \begin{tikzcd}
        \lim^{n-1}(\overline{I}\upharpoonright X)\arrow[r]&\lim^{n-1}(\overline{I}/I\upharpoonright X)\arrow[r]&\lim^n(I\upharpoonright X)\arrow[r]&\lim^n(\overline{I}\upharpoonright X).
    \end{tikzcd}
\end{center}
If $n>1$, then Proposition \ref{lim char} implies that $\lim^{n-1}(\overline{I}\upharpoonright X)=\lim^{n}(\overline{I}\upharpoonright I)=0$, so an $n$-coherent $\Phi\in K^{n-1}(\overline{I}\upharpoonright X)$ represents $0$ in $\lim^{n-1}(\overline{I}/I\upharpoonright X)$ if and only if $[\Phi]$ is in the kernel of the connecting homomorphism from $\lim^{n-1}(\overline{I}/I\upharpoonright X)$ to $\lim^n(I\upharpoonright X)$. 
But these are precisely the meanings of type $I$ and type $II$ triviality. 
The fact that type $I$ and type $II$ triviality are always equivalent for $\Om_\kappa$ systems if $n=1$ by Fact \ref{trivequiv} completes the proof. 
\end{proof}

We now connect coherent families being trivial for injective systems with additivity for all systems. 
For the proof of this, we need the following general fact, which is an immediate consequence of \cite[Proposition 2.2.1]{Groth}. 
\begin{fact} \label{limn_char}
    Suppose that $\mathcal{C},\mathcal{D}$ are abelian categories, $\mathcal{C}$ has enough injective objects, and $F\colon \mathcal{C}\to\mathcal{D}$ is left exact. 
    Suppose moreover that $\{F^i\}$ is a $\delta$-functor such that $F^0=F$ and $F^i(I)=0$ for every $0<i\leq n$ and injective object $I$. 
    Then for each $0<i\leq n$, the $F^i$ are naturally isomorphic to the derived functors of $F$. 
\end{fact}

\begin{lem} \label{triv_implies_add}
For $s<\omega$, suppose that for every injective $\Om_{\kappa}$ system $\mathcal{I}$ and each $n\leq s$, every $n$-coherent family for $\mathcal{I}$ is $n$-trivial. 
Then for each $\Om_{\kappa}$ system $\mathcal{G}$, the induced map $i_*\colon\bigoplus_\alpha\lim^s\G_\alpha\to \lim^sG$ is an isomorphism.
\end{lem}

\begin{proof}
Consider the functor $F\colon\Om_{\kappa}$ sys$\to$Ab given by $\mathcal{G}\mapsto\bigoplus_\alpha\lim\G_\alpha$. 
The $\delta$-functors 
\[\mathcal{G}\mapsto\bigoplus_\alpha\lim{}^n\G_\alpha\]
and
\[\mathcal{G}\mapsto\lim{}^nG\]
both coincide in dimension $0$ with $F$ and (using the hypothesis for the second), map injective objects to $0$ in dimensions $n\leq s$. 
Therefore, they coincide and are the derived limits of $F$ in dimensions at most $s$ by Fact \ref{limn_char}. 

It remains to show that the induced map is an isomorphism. 
By the universal property of derived limits (see \cite[Definition 2.1.4]{Weibel}), there are unique maps of $\delta$-functors $\varphi$ from $R^\bullet F$ to $\bigoplus_\alpha\lim^s\mathcal{G}_\alpha$ and $\psi$ from $R^\bullet F$ to $\lim^sG$ which are the identity in dimension $0$ as in the following diagram: 
\begin{center}
\begin{tikzcd}
R^\bullet F \arrow[rr, "\varphi"] \arrow[rd, "\psi"'] &         & \bigoplus_\alpha\lim^s\mathcal{G}_\alpha \\
                                                      & \lim^sG &                                         
\end{tikzcd}
\end{center}
Since $i_*\circ\varphi$ is a map of $\delta$-functors from $R^\bullet F$ to $\lim^sG$, we must have $i_*\circ\varphi=\psi$, so the following diagram commutes. 
\begin{center}
\begin{tikzcd}
R^\bullet F \arrow[rr, "\varphi"] \arrow[rd, "\psi"'] &         & \bigoplus_\alpha\lim^s\mathcal{G}_\alpha \arrow[ld, "i_*"'] \\
                                                      & \lim^sG &                                                            
\end{tikzcd}
\end{center}
We have shown already that $\varphi$ and $\psi$ are isomorphisms in dimension $s$, so $i_*$ is also an isomorphism in dimension $s$.
\end{proof}

The remainder of our work will be showing that coherent families are trivial for injective $\Om_\kappa$ systems after adding sufficiently many Cohen sequences. 

\section{Propagating Trivializations for Injective Systems}

In this section, we demonstrate how to propagate trivializations of a coherent family for an injective system on a set containing sufficiently many Cohen sequences to a trivialization on the whole domain. 
In section 6, we will define a sequence of cardinals $\langle \lambda_n\mid n<\omega\rangle$ and show that any $n$-coherent family indexed by a set of $\lambda_n$ many Cohen sequences is trivial once restricted to some set of $\lambda_{n-1}$ many Cohen sequences. 
The main goal of this section is to demonstrate that in the presence of a trivialization in such a large set of the sequences implies that the original $n$-coherent family must, in fact, have been trivial to begin with. 

The results of this section are a generalization of the results of \cite[Section 4]{SVHDLwLC}, where trivializations of coherent families for the system $\mathbf{A}$ propagated in a very similar manner.
The main difference in this section is that the arguments are for the class of injective $\Om_\kappa$ systems; Lemma \ref{triv_implies_add} will ultimately allow us to connect coherent families being trivial to additivity of derived limits for all systems. 
Despite arguing for the larger class of injective $\Om_\kappa$ systems, the argument in this section is extremely similar to the arguments of \cite[Section 4]{SVHDLwLC} for the system $\mathbf{A}$; we give the argument here for the sake of completeness. 
We also restate the result as a theorem of ZFC, rather than about coherent families indexed over sets of Cohen sequences as was done in \cite{SVHDLwLC}.
In this section, trivializations of coherent families will be type $I$ trivializations. 

\begin{defn}
Fix an injective $\Om_{\kappa}$ system $\mathcal{I}$, a subset $X\subseteq\omega^\kappa$, a function $g\in\omega^\kappa$, and a positive integer $n$.
We say a family $\Phi=\left\langle \phi_{\mathbf{x} }\in\overline{I}_{g\wedge \bigwedge\mathbf{x} }\,\middle|\,\mathbf{x}\in X^n\right\rangle$ is \emph{$n$-coherent below $g$ for $\mathcal{I}$} if it satisfies Definition \ref{coherent}, with $K^\bullet$ replaced with the complex where for each finite sequence $\mathbf{x} $ of elements of $X$, $I_{\bigwedge\mathbf{x}}$ is replaced with $I_{g\wedge \bigwedge\mathbf{x}}$. 
That is, 
\begin{itemize}
    \item $\phi_{\mathbf{x}}\in I_{g\wedge\bigwedge\mathbf{x}}$.
    \item whenever $\sigma\in S_n$ and $\mathbf{x}\in X^n$, $\phi_{\mathbf{x}}=\operatorname{sgn}(\sigma)\phi_{\sigma(\mathbf{x})}$.
    \item whenever $\mathbf{y}\in X^{n+1}$, $\sum_{i=0}^n(-1)^i\phi_{\mathbf{y}^i}=^*0$.
\end{itemize}
Such a family is \emph{$n$-trivial below $g$} if there exists a $\psi$ or $\Psi$ as in Definition \ref{I-trivial}, the only difference is that in the case $n=1$, $\psi\in\overline{I}_g$ and in the case $n>1$, $\psi_{x}\in \overline{I}_{g\wedge\bigwedge\mathbf{x}}$. 
That is, if $n=1$ then for each $x\in X$, $p_x\psi=^*\Phi_x$ and if $n>1$ then for each $\mathbf{x}\in X^n$, $\Phi_{\mathbf{x}}=^*\sum_{i=0}^{n-1}(-1)^i\psi_{\mathbf{x}^i}$.
As with $n$-coherent families, since $n$ and $g$ can be deduced from an $n$-coherent family below $g$, we will sometimes simply say that $\Phi$ is trivial instead of $n$-trivial below $g$. 
\end{defn}

\begin{rem}
    We emphasize that the index set of an $n$-coherent family below $g$ is all of $X^n$, not the meets of $g$ with elements of $X$. 

    If $g\in X$ and $\Phi$ is $1$-coherent below $g$ then $\Phi$ is trivialized by $\Phi_g$. 
    Typically, $X$ will consist of many Cohen sequences and so each element of $X$ will be larger than $g$ infinitely often and smaller than $g$ infinitely often. 
\end{rem}

The following proposition is crucial for the arguments of this section and is one of the main reasons we use injective objects. 
We do not know if the proposition can consistently fail for general $\Om_\kappa$ systems, but the proof relies on Corollary \ref{extend_by_0}.

\begin{prop} \label{trivial_below}
Suppose that $\mathcal{I}$ is an injective $\Om_{\kappa}$ system, $X\subseteq\omega^\kappa$, $g\in\omega^\kappa$, $n$ is a positive integer, and every $n$-coherent family for $\mathcal{I}$ indexed by $X^n$ is trivial. 
Then every $n$-coherent family below $g$ for $\mathcal{I}$ indexed by $X^n$ is $n$-trivial below $g$.
\end{prop}

\begin{proof}
Fix $\{i_{h,k}\colon I_h\to I_k\mid h\leq k\}$ as in Corollary \ref{extend_by_0}. 
Let $\Phi=\langle\varphi_{\mathbf{x}}\in I_{g\wedge \bigwedge\mathbf f}\mid\mathbf f\in X^n\rangle$ be $n$-coherent below $g$. 
Let $\Phi^*=\langle i_{g\wedge\bigwedge\mathbf{x},\bigwedge\mathbf{x}}(\varphi_{\mathbf f})\mid \mathbf{x}\in X^n\rangle$; since the $i_{x,y}$ are compatible and injective, $\Phi^*$ is $n$-coherent. 
By assumption, $\Phi^*$ is trivial, as witnessed by a single $\psi^*$ if $n=1$ or a family $\Psi^*$ if $n>1$. 
Then applying $p_{\bigwedge\mathbf{x},g\wedge\bigwedge\mathbf{x}}$ pointwise if $n>1$ or $p_{g\wedge f}$ if $n=1$, we obtain a trivialization of $\Phi$. 
\end{proof}

Our main lemma for this section is the following. 
We note that, from what we have proven above about injective systems, the proof follows mutatis mutandis from the proof of \cite[Lemma 4.2]{SVHDLwLC}; we repeat the argument for completeness. 
Note that the last condition ensures that $A$ is large enough that Lemma \ref{propagate} is not trivially false. 

\begin{lem} \label{propagate}
Suppose that 
\begin{itemize}
    \item $\mathcal{I}$ is an injective $\Om_\kappa$ system;
    \item $A\subseteq X\subseteq \omega^\kappa$;
    \item $\Phi=\langle\varphi_{\mathbf{x}}\,\mid\, \mathbf{x}\in X^n\rangle$ is an $n$-coherent family for $\mathcal{I}$;
    \item $\Phi\upharpoonright A$ is $n$-trivial for $\Phi\upharpoonright A$ the natural $n$-coherent family indexed by $A$ obtained by restricting entries in $\Phi$;
    \item For any $1\leq j\leq n-1$, every $j$-coherent family for $\mathcal{I}$ indexed by $A$ is trivial;
    \item for any infinite $E\subseteq \kappa$ and $f\colon E\to\omega$ there is a $g\in A$ such that $\{\alpha\in E\mid f(\alpha)\leq g(\alpha)\}$ is infinite.
\end{itemize}
Then $\Phi$ is $n$-trivial.
\end{lem}

\begin{proof}
We first show the $n=1$ case. 
Since $\Phi\upharpoonright A$ is trivial, there is a $\Psi\in \lim \overline{I}$ trivializing $\Phi\upharpoonright A$; we claim that $\Psi$ also trivializes $\Phi$. 
Otherwise, there is an $f\in X$ such that $E=\{\alpha\mid p_{f}\Psi(\alpha)\neq \Phi_{f}(\alpha)\}$ is infinite. 
By hypothesis on $A$, we may find some $g\in A$ such that $E'=\{\alpha\in E\mid f(\alpha)\leq g(\alpha)\}$ is infinite. 
Since $\Psi$ trivializes $\Phi\upharpoonright A$, there are only finitely many $\alpha\in E'$ where $p_f\Psi(\alpha)\neq p_{g,g\wedge f}\Phi_g(\alpha)$. 
Since $\Phi$ is coherent, there are only finitely many $\alpha\in E'$ where $p_{f,g\wedge f}\Phi_f(\alpha)\neq p_{g,g\wedge f}\Phi_g$. 
In particular, there is an $\alpha\in E'$ where $p_f\Psi(\alpha)=p_{g,g\wedge f}\Phi_g(\alpha)=p_{f,g\wedge f}\Phi_f(\alpha)$. 
Since $f(\alpha)\leq g(\alpha)$, $p_{f,g\wedge f}\Phi_f(\alpha)=\Phi_f(\alpha)$, contradicting that $\alpha\in E$.

We now present the general argument. 
For readability of notation, we will frequently remove the appropriate bonding maps in the inverse system $\overline{I}$; sums of elements which are officially are in different groups will be taken to mean first applying the homomorphisms from the data of $\overline{I}$ to project to the common meet of the indices of the group elements being added. 

Let $\mathscr{T}_1=\langle\tau_{\mathbf{x}}\mid\mathbf{x}\in A^{n-1}\rangle$ be a (type $I$) trivialization of $\Phi\upharpoonright A$. 
For each $f\in X$, let $C_{n-1}^f$ be the family below $f$ indexed by $A^{n-1}$ defined by 
\[\zeta_{\mathbf{x}}^f=\varphi_{\mathbf{x} f}+(-1)^np_{\bigwedge\mathbf{x},\bigwedge\mathbf{x}\wedge f}\tau_{\mathbf{x}}.\]
After adding appropriate bonding maps to map all group elements to $\overline{I}_{\mathbf{x}\wedge f}$, the following computation shows that $C_{n-1}^f$ is coherent: 
\[\begin{aligned}
\sum_{i=0}^{n-1}(-1)^i\zeta^f_{\boldsymbol\alpha^i}&=\sum_{i=0}^{n-1}(-1)^i\varphi_{\boldsymbol\alpha^if}+(-1)^n\sum_{i=0}^{n-1}(-1)^i\tau_{\boldsymbol\alpha^i}\\
&=^*\sum_{i=0}^{n-1}(-1)^i\varphi_{\boldsymbol\alpha^if}+(-1)^n\varphi_{\boldsymbol\alpha}\\
&=d\Phi_{\boldsymbol\alpha f}\\
&=^*0
\end{aligned}\]
where the first $=^*$ is the fact that $\mathscr{T}_1$ trivializes $\Phi\upharpoonright A$ and the last uses that $\Phi$ is $n$-coherent so that $d\Phi_{\boldsymbol\alpha f}=^*0$. 
We now define by mutual recursion, for $2\leq k\leq n$, 
\begin{itemize}
    \item $\mathscr{T}_k=\langle \tau_{\boldsymbol\alpha}^{\mathbf f}\,\mid\,\mathbf f\in X^{k-1},\boldsymbol\alpha\in A^{n-k}\rangle$ where for each $\mathbf f\in X^{k-1}$, $\langle \tau_{\boldsymbol \alpha}^{\mathbf f}\mid\boldsymbol\alpha\in A^{n-k}\rangle$ trivializes $C_{n-k+1}^{\mathbf f}$
    \item For $\mathbf f\in X^k$, families $C_{n-k}^{\mathbf f}=\langle\zeta_{\boldsymbol\alpha}^{\mathbf f}\,\mid\,\boldsymbol\alpha\in A^{n-k}\rangle$ which are $(n-k)$-coherent below $\bigwedge\mathbf f$ defined by 
    \[\zeta_{\boldsymbol\alpha}^{\mathbf f}=\varphi_{\boldsymbol\alpha\mathbf f}+(-1)^{n-k+1}\sum_{i=0}^{k-1}(-1)^i\tau_{\boldsymbol\alpha}^{\mathbf f^i}.\]
\end{itemize}
The hypothesis that every $(n-k)$-coherent family indexed by $A$ is trivial plus Proposition \ref{trivial_below} guarantees that for each $1\leq k\leq n-1$ and each $\mathbf f$ we can pass from each $C_{n-k}^{\mathbf f}$ to a trivialization, yielding $\mathscr{T}_{k+1}$. 
Therefore, it suffices to show that
\begin{enumerate}
    \item For each $\mathbf f\in X^k$, $C_{n-k}^{\mathbf f}$ is $(n-k)$-coherent
    \item $\mathscr{T}_n$ does indeed trivialize $\Phi$. 
\end{enumerate}
The first can be verified by the series of steps for every $\boldsymbol\alpha\in A^{n-k+1}$ (after applying appropriate maps to ensure the sum is taken in $\overline{I}_{\bigwedge\mathbf f\wedge\bigwedge\boldsymbol\alpha}$), taken verbatim from \cite[Claim 4.6]{SVHDLwLC}. 

\[\begin{aligned}
\sum_{i=0}^{n-k}(-1)^{i} \zeta_{\boldsymbol{\alpha}^{i}}^{\mathbf{f}}&=\sum_{i=0}^{n-k}(-1)^{i} \varphi_{\boldsymbol{\alpha}^{i} \mathbf{f}}+(-1)^{n-k+1} \sum_{i=0}^{n-k}(-1)^{i} \sum_{j=0}^{k-1}(-1)^{j} \tau_{\boldsymbol{\alpha}^{i}}^{\mathbf{f}^j}\\
&=^* \sum_{i=0}^{n-k}(-1)^{i} \varphi_{\boldsymbol{\alpha}^{i} \mathbf{f}}+(-1)^{n-k+1} \sum_{j=0}^{k-1}(-1)^{j} \zeta_{\boldsymbol{\alpha}}^{\mathbf{f}}\\
&=^{*} \sum_{i=0}^{n-k}(-1)^{i} \varphi_{\boldsymbol{\alpha}^{i} \mathbf{f}}+(-1)^{n-k+1}\left(\sum_{j=0}^{k-1}(-1)^{j}\left(\varphi_{\boldsymbol{\alpha} \mathbf{f}^{j}}+(-1)^{n-k} \sum_{\ell=0}^{k-2}(-1)^{\ell} \tau_{\boldsymbol{\alpha}}^{\left(\mathbf{f}^{j}\right)^{\ell}}\right)\right)\\
&=^{*} \sum_{i=0}^{n}(-1)^{i} \varphi_{(\boldsymbol{\alpha} \mathbf{f})^{i}}+(-1)^{n-k+1}\left(\sum_{j=0}^{k-1}(-1)^{j+n-k} \sum_{\ell=0}^{k-2}(-1)^{\ell} \tau_{\boldsymbol{\alpha}}^{\left(\mathbf{f}^{j}\right)^{\ell}}\right)\\
&=^{*}-\left(\sum_{j=0}^{k-1}(-1)^{j} \sum_{\ell=0}^{k-2}(-1)^{\ell} \tau_{\boldsymbol{\alpha}}^{\left(\mathbf{f}^{j}\right)^{\ell}}\right)\\
&=^{*}-\left(\sum_{j \leq \ell \leq k-2}(-1)^{j+\ell} \tau_{\boldsymbol{\alpha}}^{\left(\mathbf{f}^{j}\right)^{\ell}}+\sum_{\ell<j \leq k-1}(-1)^{j+\ell} \tau_{\boldsymbol{\alpha}}^{\left(\mathbf{f}^{\jmath}\right)^{\ell}}\right)\\
&=^{*}-\left(\sum_{j \leq \ell \leq k-2}(-1)^{j+\ell} \tau_{\boldsymbol{\alpha}}^{\left(\mathbf{f}^{\ell+1}\right)^{j}}+\sum_{\ell<j \leq k-1}(-1)^{j+\ell} \tau_{\boldsymbol{\alpha}}^{\left(\mathbf{f}^{j}\right)^{\ell}}\right)\\
&=^{*}-\left(\sum_{\ell<j \leq k-1}(-1)^{j+\ell+1} \tau_{\boldsymbol{\alpha}}^{\left(\mathbf{f}^{j}\right)^{\ell}}+\sum_{\ell<j \leq k-1}(-1)^{j+\ell} \tau_{\boldsymbol{\alpha}}^{\left(\mathbf{f}^{j}\right)^{\ell}}\right)\\
&=0 .
\end{aligned}\]

As in \cite{SVHDLwLC},
the fact that the functions $\tau_{\boldsymbol{\alpha}^{i}}^{\mathbf{f}^j}$ trivialize the functions $\zeta_{\boldsymbol{\alpha}}^{\mathbf{f}^{j}}$ underlies the passage from the first line to the second; replace $\zeta_{\boldsymbol{\alpha}}^{\mathbf{f}}$ with its definition to pass from the second line to the third. Nothing more than a regrouping underlies the passage from the third line to the fourth, whereupon the first sum vanishes modulo finite by the $n$-coherence of $\Phi$ at $\boldsymbol\alpha\mathbf f$. 
Simple bookkeeping converts the fifth line into the sixth, and the fact that $\left(\mathbf{f}^{j}\right)^{\ell}=\left(\mathbf{f}^{\ell+1}\right)^{j}$ for all $j \leq \ell \leq k-2$ converts the sixth line into the seventh. $\mathrm{A}$ renaming of variables in the first sum from $\ell+1$ to $j$ and $j$ to $\ell$ then yields the eighth line, whose terms all plainly cancel. 

To verify that $\mathscr{T}_n=\langle\tau^{\mathbf f}_{\{\varnothing\}}\mid \mathbf f\in X^n\rangle$ trivializes $\Phi$, note that for any $g\in A$ and $\mathbf f\in X^n$, once projected onto $\bigwedge\mathbf f\wedge g$, we have

\[\begin{aligned}
(d\mathscr{T}_n)_{\mathbf f}&=\sum_{i=0}^{n-1}(-1)^i\tau^{\mathbf f^i}\\
&=^*\sum_{i=0}^{n-1}(-1)^i\zeta^{\mathbf f^i}_g\\
&=^*\sum_{i=0}^{n-1}(-1)^i\left[\varphi_{g\mathbf f^i}+(-1)^{n}\sum_{j=0}^{n-2}(-1)^j(\tau_g^{(\mathbf f^{i})^j})\right]\\
&=^*\sum_{i=0}^{n-1}(-1)^i\varphi_{g\mathbf f^i}\\
&=^*\varphi_{\mathbf f}
\end{aligned},\]
where the first equality is definition of $d$, the second uses that $\tau^{\mathbf f^i}_{\{\varnothing\}}$ trivializes $\langle\zeta_g^{\mathbf f^i}\mid g\in A\rangle$, the third is expanding the definition of $\zeta$, the fourth is cancelling $\tau_g^{(\mathbf{f}^i)^j}$ with the corresponding term with indices in removed in the opposite order, and the last equality is the coherence of $\Phi$ at $g\mathbf f$. 
Now, if $d\mathscr{T}_n(\mathbf f)$ and $\varphi_{\mathbf f}$ disagreed infinitely often, let $E=\{\alpha\mid d\mathscr{T}_n(\mathbf f)(\alpha)\neq\varphi_{\mathbf f}(\alpha)\}$. 
By hypothesis, we may find an $g\in A$ such that $\{\alpha\in E\mid\bigwedge\mathbf f(\alpha)\leq g(\alpha)\}$ is infinite. 
But then $d\mathscr{T}_n(\mathbf f)$ disagrees with $\varphi_{\mathbf f}$ infinitely often, even after projecting onto $g\wedge\bigwedge\mathbf f$, contrary to the above calculation. 
\end{proof}

The relevance of Lemma \ref{propagate} to Cohen forcing is the following:

\begin{prop} \label{unbdd}
Suppose $G$ is $\operatorname{Fn}(\chi\times\kappa,\omega)$ generic and suppose $a\subseteq\chi$ is in $V[G]$ and has cardinality at least $\kappa^+$. 
Let $A=\{f_\alpha\mid \alpha\in a\}$. 
Then $A$ satisfies the last clause of Lemma \ref{propagate}. 
\end{prop}

\begin{proof}
Work in $V[G]$ and suppose $E\subseteq\kappa$ is infinite. 
By shrining $E$ if necessary, we may assume $E$ is countable. 
By the countable chain condition, there is a $W\subseteq\chi$ in $V$ such that $|W|\leq\kappa$ and $E\in V[G_W]$. 
By mutual genericity, the sequence indexed by any $\beta\in a\setminus W$ will do. 
\end{proof}

The plan for the remainder of the argument will consist of defining a sequence $\langle \lambda_n\mid n<\omega\rangle$ of cardinals such that any $n$-coherent family indexed by a set containing at least $\lambda_n$ many of the Cohen sequences is trivial once restricted to a set containing at least $\lambda_{n-1}$ many Cohen sequences. 
Lemma \ref{propagate} will then allow us to inductively conclude that every $n$-coherent family indexed by a set containing at least $\lambda_n$ many Cohen sequences is trivial. 
The first step factors through a partition principle, analogous to the one defined in \cite[Definition 3.2]{DAHDL}. 

\section{A Local Partition Result in Cohen Extensions}

We now extract a partition result in Cohen extensions which is essentially proven in section 6 of \cite{SVHDLwLC} and will allow us to create type $II$ trivializations of families indexed by very large sets of Cohen sequences on sets with fewer but still sufficiently many Cohen sequences. 
We first make a preliminary definition. 

\begin{defn} \label{str}
A \emph{string on $[X]^{\leq n}$} is a sequence $\langle a_i\mid 0\leq i\leq m\rangle$ of finite subsets of $X$ such that 
\begin{itemize}
    \item $m\leq n$;
    \item for each $i$, $a_i\subsetneq a_{i+1}$;
    \item for $i<m$, $|a_{i+1}|=|a_i|+1$;
    \item $|a_m|=n$.
\end{itemize}
We denote the set of maximal strings on $X$ as $X^{\bb n}$.
If $\vec a$ is a maximal string and $F\colon [X]^{\leq n}\to X$ is a function, we write $F^*(\sigma)=\langle F(a_1),\ldots,F(a_n)\rangle$. 
\end{defn}

\begin{lem} \label{LPH}
    Let $\kappa$ be a cardinal, let $\lambda$ be regular such that for every $\nu<\lambda$ $\nu^\kappa<\lambda$, and let $n\geq2$. 
    Let $\mu=\sigma(\lambda^+,2n-1)$ as defined in Definition \ref{sigma}, let $\chi\geq\mu$ be a cardinal, and let $\mathbb{P}=\operatorname{Fn}(\chi\times\kappa,\omega)$. 
    The following holds in $V^{\mathbb{P}}:$
    Suppose that $X\subseteq\omega^\kappa$ consists of $\mu$ many Cohen sequences added by $\mathbb{P}$ and $c\colon[X]^n\to\kappa$. 
    Then there is an $A\subseteq X$ of cardinality $\lambda$, a finite set $S\subseteq\kappa$, functions $\{F_\alpha\colon[A]^{\leq n}\setminus\{\varnothing\}\to X\mid\alpha\in\kappa\setminus S\}$, and a $\beta<\kappa$ such that for
    \begin{enumerate}
        \item for each $\alpha\in\kappa\setminus S$ and $a\in A$, $F_\alpha(\{a\})=a$. \label{id}
        \item for each $\alpha\in\kappa\setminus S$ and $\tau\subsetneq\sigma$, $F_\alpha(\tau)(\alpha)<F_\alpha(\sigma)(\alpha)$. \label{inc}
        \item For each $\alpha\in\kappa\setminus S$ and any maximal string $\langle a_i\mid i\leq n\rangle$ on $[A]^{\leq n}$, 
        \[c(\{F_\alpha(a_i)\mid 1\leq i\leq n\})=\beta.\] \label{PH}
        \item For any string $\langle a_i\mid i\leq m\rangle$ on $[A]^{\leq n-1}$ with $|a_0|\geq2$, the set
        \[\left\{c(a_0\cup\{F_\alpha(a_i)\mid 0\leq i\leq m\})\mid \alpha\in\kappa\setminus S\right\}\]
        is finite. 
        
        Note this set has size $n$ since $|a_0|=n-1-m$ and \[\max_{f\in a_0}f(\alpha)<F_\alpha(a_0)(\alpha)<F_\alpha(a_1)(\alpha)<\ldots<F_\alpha(a_m)(\alpha)\]
        by \ref{id} and \ref{inc} so all entries are distinct. \label{lf}
        
    \end{enumerate}
\end{lem}

\begin{rem}
It is perhaps worth comparing the principle of Lemma \ref{LPH} above to the principle $\PH_n$ defined in \cite[Definition 3.2]{DAHDL}. 
The first difference worth noting is the change of domain: while $\PH_n$ is a partition principle defined over all the real numbers, the principle of Lemma \ref{LPH} is only about large sets of Cohen sequences. 
The reason for the change is the requirement of item \ref{lf}, which makes extension of the functions $F_\alpha$ from $A$ to all of $X$ much more difficult. 
Individually, item \ref{id} may as well have been required for $\PH_n$ on a cofinal set, \ref{inc} and \ref{PH} are direct analogues of what appears in $\PH_n$, and item \ref{lf} is entirely new but automatic for $\PH_n$ since the principle gave only one function $F$. 
\end{rem}

\begin{proof}
Fix a condition $p$ such that
\begin{itemize}
    \item $p\Vdash|\{\alpha<\chi\mid\dot f_\alpha\in\dot X\}|=\mu$ and
    \item $p\Vdash \dot c\colon [\dot X]^n\to \kappa$. 
\end{itemize}
Let $Y$ be the set of all $\alpha<\chi$ for which there is a condition $p_\alpha\leq p$ such that $p_\alpha\Vdash\dot f_\alpha\in\dot X$; observe that $|Y|\geq\mu$ by assumption. 
For each $\alpha\in Y$, fix such a condition $p_\alpha$. 
Since $\mathbb{P}$ has $\mu$ as a precalibre, there exists a $Y'\subseteq Y$ of cardinality $\mu$ such that every finite subset of $\{p_\alpha\mid \alpha\in Y'\}$ has a lower bound. 
For each $a\in[Y']^{n}$, let $\langle q_{a,\ell}\mid \ell<\omega\rangle$ enumerate a maximal antichain of conditions below $\bigcup_{\alpha\in a}p_\alpha$ deciding the value of $\dot c(\{\dot f_\alpha\mid\alpha\in a\})$ to be some $\beta_{a,\ell}$. 
Let \[u(a,\ell)=\{\gamma<\chi\mid\dom(q_{a,\ell}\cap(\{\gamma\}\times\kappa)\neq\varnothing)\}.\]
For each $b\in[Y']^{2n-1}$ let $v_b=\bigcup\{u(a,\ell)\,\mid\,a\in[b]^n,\ell<\omega\}$. 
Define $F\colon[Y']^{2n-1}\to H(\kappa^+)$ as follows. 
First, for $b\in[Y']^{2n-1}$ and $\mathbf{m}\in[2n-1]^{n}$ and $\ell<\omega$, let
\[w_{\mathbf{m},\ell}^b=\{\eta<\operatorname{otp}(v_b)\,\mid\,v_b(\eta)\in u(b[\mathbf{m}],\ell)\}.\]
Note here we have identified sets of ordinals with an increasing sequence of ordinals to define $u(b[\mathbf{m}])$ and $v_b(\eta)$.

Then set
\[F(b)=\langle\langle \overline{q}_{b[\mathbf{m}],\ell},w_{\mathbf{m},\ell}^b,\beta_{b[\mathbf{m}],\ell}\rangle\,\mid\,\mathbf{m}\in[2n-1]^n,\ell<\omega\rangle,\]
where we recall that $\overline{q}$ is the collapsed version of $q$ as defined in section \ref{Prelim}.
By Fact \ref{DSys_lem}, there exists $H\in[Y']^{\lambda^+}$ such that
\begin{itemize}
    \item $F$ is constant on $[H]^{2n-1}$, taking value $\langle\langle \overline{q}_{\mathbf{m},\ell},w_{\mathbf{m},\ell},\beta_{\mathbf{m},\ell}\,\mid\,\mathbf{m}\in[2n-1]^{n},\ell<\omega\rangle$;
    \item $\langle v_b\,\mid\,b\in[H]^{2n-1}\rangle$ is a uniform $2n-1$ dimensional $\Delta$-system as witnessed by $\rho$ and $\langle\mathbf{r_m}\,\mid\,\mathbf{m}\subseteq 2n-1\rangle$.
\end{itemize}
By taking an initial segment if necessary, we may assume that $H$ has order type $\lambda^+$. 
Let $\langle v_{a,k}\,\mid\,a\in[H]^{<2n-1}, k\leq|a|\rangle$ be given by Lemma \ref{k_addable} whenever there is an $\alpha\in H$ which is $k$-addable for $a$. 

\begin{claim} \label{restrict}
Let $a\in[H]^{n-1}$, $k\leq n-1$, and $\ell<\omega$ and suppose that $\alpha,\alpha'\in H$ are both $k$-addable for $a$. 
Then $q_{a\cup\{\alpha\},\ell}\upharpoonright(v_{a,k}\times\kappa)=q_{a\cup\{\alpha'\},\ell}\upharpoonright(v_{a,k}\times\kappa)$. 
\end{claim}
\begin{proof}
We can assume that $\alpha\neq\alpha'$. 
Fix $c\in[H]^{n-1}$ with $\min c>\max(a\cup\{\alpha,\alpha'\})$ and let $b=a\cup\{\alpha\}\cup c$ and $b'=a\cup\{\alpha'\}\cup c$. 
Then $b$ and $b'$ are aligned with $\mathbf{r}(b,b')=2n-1\setminus\{k\}$, so $v_b$ and $v_{b'}$ are aligned with $\mathbf{r}(v_b,v_{b'})=\mathbf{r}_{2n-1\setminus\{k\}}$. Moreover, $v_{a,k}=v_b[\mathbf{r}_{n\setminus\{k\}}]=v_{b'}[\mathbf{r}_{n\setminus\{k\}}]$. 
By homogeneity of $H$, $F(b)=F(b')$ so $u(a\cup\{\alpha\},\ell)$ and $u(a\cup\{\alpha'\},\ell)$ occupy the same relative positions in $v_b$ and $v_{b'}$ respectively since $F$ records the positions of each $w^b_{\mathbf{m},\ell}$. 
Since $v_b$ and $v_{b'}$ are aligned, $u(a\cup\{\alpha\},\ell)$ and $u(a\cup\{\alpha'\},\ell)$ are aligned and since $v_{a\cup\{\alpha\}}\cap v_{a\cup\{\alpha'\}}$, $u(a\cup\{\alpha\},\ell)\cap u(a\cup\{\alpha'\},\ell)=v_{a,k}\cap u(a\cup\{\alpha\},\ell)=v_{a,k}\cap u(a\cup\{\alpha'\},\ell)$ as we chose $v_{a,k}$ to satisfy the conclusion of Lemma \ref{k_addable}. 
Moreover, by homogeneity of $H$, since the partition $F$ recorded the collapsed version of conditions, $\overline{q}_{a\cup\{\alpha\},\ell}=\overline{q}_{a\cup\{\alpha'\},\ell}=\overline{q}_{n,\ell}$. 
Thus, if $(\beta,\gamma)\in v_{a,k}\times\kappa$ is in the domain of $q_{a\cup\{\alpha\},\ell}$ with $\beta=u(a\cup\{\alpha\},\ell)(i)$ then it is also in the domain of $q_{a\cup\{\alpha'\},\ell}$ and the value of each is $\overline{q}_{n,\ell}(i,\gamma)$ and the same holds in reverse. 
\end{proof}

We now note that the values of $\overline{q}_{\mathbf{m},\ell}$ and $\beta_{\mathbf{m},\ell}$ are independent of $\mathbf{m}\in[2n-1]^n$. 
To see this, given $b\in[H]^{2n-1}$ and $\mathbf{m}\in [2n-1]^n$, there is a $b^*\in[H]^{2n-1}$ such that $b[\mathbf{m}]=b^*[n]$. 
Then
\[\overline{q}_{\mathbf{m},\ell}=\overline{q}_{b[\mathbf{m}],\ell}=\overline{q}_{b^*[n],\ell}=\overline{q}_{n,\ell}\]
and similarly for $\beta_{\mathbf{m},\ell}$. 

Let $\overline{q}_\ell$ and $\beta_\ell$ be the constant values of $\overline{q}_{\mathbf{m},\ell}$ and $\beta_{\mathbf{m},\ell}$ respectively. 
Now, let $H_0,...,H_{n-1}$ be disjoint subsets of $H$ such that $\operatorname{otp}(H_k)=\lambda$ and $H_k<H_{k'}$ for all $k<k'<n$. 
For each $k< n$ let $\delta_k=\min(H_k)$. 
For $a\in [H_0]^{\leq n-1}$, let $d_a=a\cup\{\delta_i\mid |a|\leq i< n\}$ and let $d^+=\{\delta_k\,\mid\,1\leq k< n\}$. 
For $\alpha\in H_0$, we write $d_{\alpha}$ to mean $d_{\{\alpha\}}=\{\alpha\}\cup d^+$. 
Let $q=q_{d_\varnothing,0}$ and let $\dot B$ be a $\mathbb{P}$-name for $\{\alpha\in H_0\,\mid\,q_{d_{\alpha},0}\in \dot G\}$. 
We claim that $q$ forces that there are functions $\dot F_j$ as in the lemma on the set $A=\{\dot f_\alpha\mid\alpha\in\dot B\}$. 
We first claim that $\dot B$ is forced to have cardinality $\lambda$. 
By regularity of $\lambda$, it suffices to show the following. 
\begin{claim}
$q$ forces that $\dot B$ is unbounded in $H_0$. 
\end{claim}
\begin{proof}
Fix an arbitrary $\gamma\in H_0$ and an arbitrary $r\leq q$. 
By choice of $v_{d^+,0}$ and Lemma \ref{k_addable}, $\{v_{d_{\alpha}}\,\mid\,\alpha\in H_0\setminus\gamma\}$ is an infinite 1-dimensional $\Delta$-system with root $v_{d^+,0}$. 
Therefore, there exists an $\delta\in H_0\setminus\gamma$ such that $v_{d_\delta}\setminus v_{d^+,0}$ is disjoint from $\{\xi\mid\exists\zeta((\xi,\zeta)\in\dom(r))\}$. 
By Claim \ref{restrict} with $a=d^+$, $k=\ell=0$, $\alpha=\delta_0$, and $\alpha'=\delta$, $q_{d_{\delta},0}\upharpoonright(v_{d^+,0}\times\kappa)=q_{d_\varnothing,0}\upharpoonright(v_{d^+,0}\times\kappa)$. 
Since $r$ extends $q_{d_\varnothing,0}=q$, it is compatible with $q_{d_\delta,0}$. 
Since $\gamma$ was arbitrary, the claim is established.
\end{proof}

The $\delta_i$ are currently used as placeholders which must be replaced by other ordinals. 
The ultimate goal will be to define each $F_\alpha(a)$ to be a member of $H_{|a|-1}$; defining $F_\alpha(a)$ will be a matter of finding a sufficient replacement for $\delta_{|a|-1}$. 
The following claim is the main mechanism by which we will replace the $\delta_i$. 

\begin{claim} \label{exchange}
Suppose that $c_0,c_1\in[H]^{n-1}$, $k_0,k_1\leq n-1$, $\alpha\in H$ and $\ell_0,\ell_1<\omega$ are such that
\begin{itemize}
    \item $\alpha$ is $k_0$-addable for $c_0$
    \item $\alpha$ is $k_1$-addable for $c_1$
    \item $q_{c_0\cup\{\alpha\},\ell_0}$ and $q_{c_1\cup\{\alpha\},\ell_1}$ are compatible. 
\end{itemize}
Then for every $\alpha'\in H$ such that $\alpha'$ is $k_0$-addable for $c_0$ and $k_1$-addable for $c_1$, $q_{c_0\cup\{\alpha'\},\ell_0}$ and $q_{c_1\cup\{\alpha'\},\ell_1}$ are compatible. 
\end{claim}
\begin{proof}
We may assume that $\alpha\neq\alpha'$. 
Let $c=c_0\cup c_1$; then $|c|\leq 2n-2$. 
Let $b\in[H]^{2n-1}$ be a (possibly trivial) end-extension of $c\cup\{\alpha\}$. 
Let $i^*$ be such that $b(i^*)=\alpha$ and let $b'=(b\setminus\{\alpha\})\cup\{\alpha'\}$.
Then $b$ and $b'$ are aligned with $\mathbf{r}(b,b')=2n-1\setminus\{i^*\}$. 
Let $\mathbf{m}_0,\mathbf{m}_1$ be such that $c_i\cup\{\alpha\}=b[\mathbf{m}_i]$; then $c_i\cup\{\alpha'\}=b'[\mathbf{m}_i]$. 

Suppose that $(\gamma',\delta)\in\dom(q_{c_0\cup\{\alpha'\},\ell_0})\cap\dom(q_{c_1\cup\{\alpha'\}},\ell_1)$. 
Let $\gamma'=u_{b'}(\eta)$. 
Since $F(b)=F(b')$, the relative position of $u(c_i\cup\{\alpha\},\ell_i)$ within $v_b$ is the same as the relative position of $u(c_i\cup\{\alpha'\},\ell_i)$ within $v_{b'}$. 
Moreover, $\overline{q}_{c_i\cup\{\alpha\},\ell_i}=\overline{q}_{c_i\cup\{\alpha'\},\ell_i}$. 
In particular, if $\gamma=u_b(\eta)$ then $(\gamma,\delta)\in\dom(q_{c_0\cup\{\alpha\},\ell_0})\cap\dom(q_{c_1\cup\{\alpha\}},\ell_1)$ and $q_{c_i\cup\{\alpha\},\ell_i}(\gamma,\delta)=q_{c_i\cup\{\alpha'\},\ell_i}(\gamma',\delta)$. 
Since $q_{c_0\cup\{\alpha\},\ell_0}$ and $q_{c_1\cup\{\alpha\},\ell_1}$ are compatible, 
\[q_{c_0\cup\{\alpha'\},\ell_0}(\gamma',\delta)=q_{c_0\cup\{\alpha\},\ell_0}(\gamma,\delta)=q_{c_1\cup\{\alpha\},\ell_1}(\gamma,\delta)=q_{c_1\cup\{\alpha'\},\ell_1}(\gamma',\delta),\]
as desired. 
\end{proof}
We now turn more directly to constructing the objects in the statement of the lemma. 
If $r\Vdash a\in[\dot B]^{\leq n}$ then for each $\alpha\in a$, $r\Vdash q_{d_\alpha,0}\in\dot G$ by definition of $\dot B$. 
In particular, since $q_{d_\alpha,0}\leq p_\alpha$ by definition of $q_{d_\alpha,0}$, we must have $r\Vdash\forall\alpha\in a(p_\alpha\in\dot G)$. 
In particular, since $\langle q_{d_a,\ell}\mid \ell<\omega\rangle$ enumerates a maximal antichain below $p_a$ and $r\leq p_a$, we may let $\dot\ell_{a}$ be a name such that $r$ forces $\dot\ell_a$ is the unique $\ell$ such that $q_{d_a,\ell}\in G$. 
Let $\dot\beta$ be a name for $\dot\beta_{\dot\ell_\varnothing}$.
Let $\dot S_{a}$ be a name for 
\[\{\gamma<\kappa\,\mid\,(\delta,\gamma)\in\dom(q_{d_{\dot a},\ell_{a}})\text{ for some }\delta<\chi\},\]
observing that $\dot S_{a}$ is forced to be finite. 
Moreover, by homogeneity of $H$ and the fact that $a\subseteq H_0\subseteq H$, $\dot S_{a}$ is simply 
\[\{\gamma<\kappa\,\mid\,(\delta,\gamma)\in\dom(\overline{q}_{\ell_{a}})\text{ for some }\delta<\kappa\}.\]
Let $\dot S=\dot S_\varnothing$.
The key idea for ensuring $q$ forces (\ref{lf}) of Lemma \ref{LPH} in $V[G]$ is that $q$ forces that whenever $\langle a_i\mid i\leq m\rangle$ is a string on $[A]^{\leq n-1}$ with $a_0\neq\varnothing$ and $\alpha\not\in S_{a_0}$ then $c(a_0\cup\{F_\alpha(a_i)\mid i\leq m\})=\beta_{\dot\ell_{a_0}}$ and that if $a_0=\varnothing$ and $\alpha\not\in S$ then $c(\{F_\alpha(a_i)\mid 1\leq i\leq n\})=\beta$.

We will show that $q$ forces the existence of ordinals
\[\langle \varepsilon_a^\alpha\in H_{|a|-1}\,\mid\,a\in[B]^{\leq n}\setminus\{\varnothing\}, \alpha\in\kappa\setminus \dot S\rangle\]
with the intention of setting $F_\alpha(\{f_\alpha\mid\alpha\in a\})=f_{\varepsilon_a^\alpha}$. 
The ordinals $\varepsilon^\alpha_a$ will satisfy a few additional properties; to state these, we need some further notation. 
If $\vec a=\langle a_i\mid 0\leq i\leq m\rangle$ is a string on $[A]^{\leq j}$ for some $j\leq n$, we denote
\[d_{\vec a}^\alpha=a_0\cup\{\varepsilon_{a_i}^\alpha\,\mid\,0\leq i\leq m\}\cup\{\delta_k\mid j\leq k<n-1\}\]
if $|a_0|\geq2$ and 
\[d_{\vec a}^\alpha=\{\varepsilon_{a_i}^\alpha\mid 1\leq i\leq m\}\cup\{\delta_k\mid j\leq k<n\}\]
if $a_0=\varnothing$.
We will ensure $d_{\vec a}^\alpha$ is always in $[H]^n$
\begin{claim} \label{LPH_main}
$q$ forces that there exist ordinals 
\[\langle\varepsilon_a^\alpha\,\mid\,a\in[\dot B]^{\leq n}\setminus\varnothing, \alpha\in\kappa\setminus \dot S\rangle\]
satisfying
\begin{enumerate}
    \item $\varepsilon_a^\alpha\in H_{|a|-1}$ for all nonempty $a\in[\dot B]^{\leq n}$ and $\alpha\in \kappa\setminus \dot S$; \label{hom}
    \item $\varepsilon_{\{\beta\}}^\alpha=\beta$ for all $\beta\in \dot B$ and all $\alpha\in\kappa\setminus \dot S$; \label{idy}
    \item For $\varnothing\subsetneq a\subsetneq b$ and all $\alpha\in\kappa\setminus \dot S$, $\dot f_{\varepsilon_a^\alpha}(\alpha)<\dot f_{ \varepsilon_b^\alpha}(\alpha)$; \label{incr}
    \item For every $m$ with $1\leq m\leq n$, every string $\vec a$ on $[\dot B]^{\leq m}$ with $|a_0|\neq 1$ and every $\alpha\not\in \dot S_{a_0}$,
    \[q_{d_{\vec a}^\alpha,\dot\ell_{a_0}}\in\dot G.\] 
    \label{cons}
\end{enumerate}
\end{claim}
\begin{proof}
Note that there is a natural candidate for $\epsilon_a^\alpha$, namely $\delta_{|a|-1}$: setting $\varepsilon^\alpha_a=\delta_{|a|-1}$ would make $d^\alpha_{\vec a}=d^\alpha_{\vec a\setminus a}$ for any string $\vec a$ on $a$. In particular, so long as we had ensured properties (\ref{hom})-(\ref{cons}) for sets of size $|a|-1$, setting $\epsilon_a^\alpha=\delta_{|a|-1}$ would maintain all of (\ref{hom}), (\ref{idy}), and (\ref{cons}). 
The only problem is in (\ref{incr}). 
The main difficulty will be in maintaining (\ref{cons}) while obtaining (\ref{incr}).
We will do this by showing that any condition below $q$ can be extended to force that we may find a suitable replacement for $\delta_{|a|-1}$ to go from $d_{\vec \sigma\setminus a}^\alpha$ to $d_{\vec\sigma}^\alpha$ for each string $\vec\sigma$ on $a$. 

The construction is by induction on $|a|$. 
If $a$ and $b$ have the same size, they could both appear in one of the strings required in (\ref{cons}); however, such strings would also necessarily involve sets of larger size (at least that of $|a\cup b|$). then there are no requirements on building $\varepsilon^\alpha_a$ and $\varepsilon^\alpha_b$ that mention $a$ and $b$ but not $a\cup b$ (of larger size). 
In particular, so long as we can maintain these properties inductively on size, we may handle a single set $a$ at a time. 
Moreover, the requirements are independent of $\alpha$, so we may again handle one $\alpha$ at a time.
Our strategy will be to show by induction on $|a|$ that once names have been fixed for any subset of $a$ to meet the requirements involving sets of size $<|a|$, any $r\leq q$ may be extended to force some value of $\varepsilon^\alpha_a$ satisfying the requirements involving sets of size $<|a|$ together with $a$.

If $|a|=1$, so $a=\{\beta\}$ for some $\beta$, then setting $\varepsilon_a^\alpha=\beta$ will do; property (\ref{hom}) follows since $B\subseteq H_0$, property (\ref{idy}) is immediate, property (\ref{incr}) is vacuously satisfied, and property $(4)$ also follows because $B$ was defined to be the collection of $\beta\in H_0$ with $q_{\beta,0}\in\dot G$, noting that $\ell_{\{\beta\}}=0$ and the only strings that are relevant for (\ref{cons}) are those of the form $\langle \varnothing,\{\beta\}\rangle$. 

Suppose that $r\leq q$ forces that $\dot a\in[B]^m$. We may assume that $r$ decides the value of $\dot S$ to be some $S$, so fix an $\alpha\in\kappa\setminus S$. 
By extending $r$ if necessary, we may assume that $r$ decides 
\begin{itemize}
    \item $\dot a=\check a$ for some $a\subseteq H_0$
    \item the values of $\varepsilon_{a'}^\alpha$ and $\ell_{a'}$ for each $a'\subsetneq a$
    \item the value of $\ell_a$
    \item the value of $\dot f_{\varepsilon_{a'}^\alpha}(\alpha)$ for each $a'\subsetneq a$
\end{itemize}

The strings in which we need to replace $\delta_i$ for some value of $i$ are all among the strings on $[a]^{\leq|a|-1}$. 
For such a string, the value to replace will be $\delta_{|a|-1}$; the only exception where $\delta_{|a|-1}$ does not appear and therefore does not need to be replaced is when $|a|=n$ and $a_0\neq\varnothing$; we will say strings where $\delta_{|a|-1}$ must be replaced are \emph{relevant}.
The goal is to show that $\delta_{|a|-1}$ may be replaced with some $\varepsilon\in H_{|a|-1}$ such that there is some $s\leq r$ which forces that setting $\varepsilon^\alpha_a=\varepsilon$ maintains properties (\ref{hom})-(\ref{cons}); in particular, we maintain (\ref{cons}) while adding (\ref{incr}).

If $\vec a$ is a string on $[a]^{\leq|a|-1}$, then since $r\Vdash q_{d_{\vec a}^\alpha,\ell_{a_0}}\in\dot G$, we must have $r\leq q_{d_{\vec a}^\alpha,\ell_{a_0}}$. 
Now, note that if $\vec a$ and $\vec b$ are relevant strings on $[a]^{\leq|a|-1}$, there is an $\beta$ (namely $\delta_{|a|-1}$) such that $q_{d^\alpha_{\vec a}\setminus\delta_{|a|-1}\cup\{\beta\},\ell_{a_1}}$ and $q_{d^\alpha_{\vec a}\setminus\delta_{|a|-1}\cup\{\beta\},\ell_{b_1}}$ are compatible. 
Therefore, by Claim \ref{exchange}, for every $\varepsilon\in H_{|a|-1}$, $q_{d_{\vec a}\setminus\delta_{|a|-1}\cup\{\varepsilon\},\ell_{a_1}}$ and $q_{d_{\vec a}\setminus\delta_{|a|-1}\cup\{\varepsilon\},\ell_{b_1}}$ are compatible. 
Now, by choice of the $v_{a',k}$, note that for every string $\vec a$, $\{v_{d_{\vec a}\setminus\delta_{|a|-1}\cup\{\varepsilon\},\ell_{a_0}}\,\mid\,\varepsilon\in H_{|a|-1}\}$ is an infinite $\Delta$-system with root $v_{d_{\vec a}\setminus\{\delta_{|a|-1}\},|a|-1}$, so there is an $\varepsilon $ such that $v_{d_{\vec a}\setminus\delta_{|a|-1}\cup\{\varepsilon\},\ell_{a_0}}\setminus v_{d_{\vec a}\setminus\{\delta_{|a|-1}\},|a|-1}$ is disjoint from $\{\gamma<\chi\mid\exists\delta(\delta,\gamma)\in\dom(r)\}$ for every such $\vec a$. 
By Claim \ref{restrict}, note that $q_{d_{\vec a}\setminus\delta_{|a|-1}\cup\{\varepsilon\},\ell_{a_0}}\upharpoonright (v_{d_{\vec a}\setminus\{\delta_{|a|-1}\},|a|-1}\times\kappa)=q_{d_{\vec a},\ell_{a_0}}\upharpoonright (v_{d_{\vec a}\setminus\{\delta_{|a|-1}\},|a|-1}\times\kappa)$.
Therefore, 
\[s_0=r\cup \bigcup\left\{q_{d^\alpha_{\vec a}\setminus\{\delta_{|a|-1}\}\cup\{\varepsilon\},\ell_{a_0}}\,\middle|\,\vec a\text{ is a string on }[a]^{\leq|a|-1},\alpha\not\in S_{a_0}\right\}\]
is a condition extending $r$. 
By choice of $\varepsilon$, there is no $\beta$ with $(\varepsilon,\beta)\in\dom(r)$ and by choice of $S_{a_0}$ there is no $\vec a$ and $\beta$ with $\alpha\not\in S_{a_0}$ but $(\beta,\alpha)\in\dom(q_{d^\alpha_{\vec a}\setminus\{\delta_{|a|-1}\}\cup\{\varepsilon\},\ell_{a_0}})$. 
In particular, $(\varepsilon,\alpha)\not\in\dom(s_0)$. 
Therefore, we extend $s_0$ to define $\dot f_\varepsilon(\alpha)$ to be larger than $\dot f_{\varepsilon_{a'}^\alpha}(\alpha)$ for any $a'\subsetneq a$. 
This $s$ is readily verified to force that $\varepsilon_a^\alpha=\varepsilon$ satisfies all of the desired conditions. 
\end{proof}
To verify the conclusions of Lemma \ref{LPH}, we work in $V[G]$. 
For $a\in [A]^{\leq n}\setminus\emptyset$, set $F_{\alpha}(\{f_\gamma\,\mid\,\gamma\in a\})=f_{\varepsilon_a^\alpha}$ as constructed in Claim \ref{LPH_main}. 
The Conclusion (\ref{id}) of Lemma \ref{LPH} is immediate from requirement (\ref{idy}) on $\varepsilon_{\{\beta\}}^\alpha$ from Claim \ref{LPH_main}. 
Conclusion (\ref{inc}) of Lemma \ref{LPH} is immediate from requirement (\ref{incr}) of Claim \ref{LPH_main}. 
For the third conclusion, if $\vec a$ is a maximal string on $[A]^{\leq n}$ then since $q_{d_{\vec a}^\alpha,0}\in\dot G$ by item (\ref{hom}) of Claim \ref{LPH_main}, 
\[c\left(\{F_\alpha(a_i)\mid 1\leq i\leq n\}\right)=\beta,\]
where $\beta=\beta_{\ell_0}$ was defined in the paragraph proceeding Claim \ref{exchange}.
For conclusion (\ref{lf}) of Lemma \ref{LPH}, if $\vec a=\langle a_i\mid i\leq m\rangle$ is a string on $[A]^{\leq n-1
}$ with $a_0\neq\varnothing$, then since $q_{d_{\vec a}^\alpha,\ell_{a_0}}\in\dot G$ whenever $\alpha\not\in S_{a_0}$ by item (\ref{hom}) of Claim \ref{LPH_main}, 
\[c\left(a_0\cup\{F_\alpha(a_i)\mid 0\leq i\leq m\}\right)=\beta_{a_0}.\]
\end{proof}

\section{From Partition to Trivialization}
\subsection{Trivialization on a Large Set}
Suppose that $\mathcal{G}$ is an $\Om_{\kappa}$ system and let $\Phi$ be an $n$-coherent family for $\mathcal{G}$ indexed over $X^n$. 
\begin{defn}
$c_\Phi\colon[X]^{n+1}\to[\kappa]^{<\omega}$ is defined by setting $c_\Phi(\mathbf{x})$ to be the support of $d\Phi(\mathbf{x})$; that is, $c_\Phi(\mathbf{x})=\{\alpha<\kappa\mid d\Phi(\mathbf{x})(\alpha)\neq0\}$ for some enumeration $\vec x$ of $\mathbf{x}$. 
Note that $d\Phi$ is indexed by sequences, so we do need to pass from a set to a sequence. 
However, the support of $d\Phi$ is independent of the choice of enumeration of $\mathbf{x}$.
\end{defn}

We now demonstrate how to produce a trivialization of $\Phi\upharpoonright A$ for some large set of Cohen sequences $A$ from the conclusions of Lemma \ref{LPH} applied to $c_\Phi$. 

\begin{lem} \label{large_triv}
Suppose that $\lambda$ is $<\kappa^+$-inaccessible and $\mu=\sigma(\lambda^+,2n+1)$. 
Then for any $\chi\geq\mu$, the following is forced by $\mathbb{P}=\operatorname{Fn}(\chi\times\kappa,\omega)$:
Suppose that $\mathcal{G}$ is an $\Om_{\kappa}$ system and $\Phi$ is an $n$-coherent family corresponding to $\mathcal{G}$ indexed by $X^n$ with $X$ containing at least $\mu$ many of the Cohen sequences added by $\mathbb{P}$. 
Then there is an $A\subseteq X$ containing at least $\lambda$ many Cohen sequences added by $\mathbb{P}$ such that $\Phi\upharpoonright A$ is type $II$ trivial. 
\end{lem}

\begin{proof}
Let $A,S,F_\alpha,T$ be as in the conclusions of Lemma \ref{LPH} applied to $c_\Phi$; note that since $c_\Phi$ colors using finite sets, we use $T$ rather than $\beta$ to denote the fixed color of $c_\Phi$ on $F_\alpha$ applied to maximal strings.  
Let $S^*=S\cup T$. 
In the $\lim{}^1$ case the result is fairly easy: for any $x,y\in A$ and $\alpha\not\in S^*$, after applying appropriate maps $p$, 
\[\begin{aligned}
    \Phi_x(\alpha)-\Phi_y(\alpha)&=(\Phi_x(\alpha)-\Phi_{F_\alpha(x,y)})+(-\Phi_y(\alpha)+\Phi_{F_\alpha(x,y)}(\alpha))\\
    &=d\Phi_{x,F_{\alpha}(x,y}(\alpha)-d\Phi_{y,F_\alpha(x,y)}(\alpha)\\
    &=0,
\end{aligned}\]
where the first equality is adding $0$ (noting that $F_\alpha(x,y)(\alpha)\geq x(\alpha),y(\alpha)$), the second equality is the definition of $d$, and the third is that by choice of $F_\alpha$, $c_\Phi(x,F_\alpha(x,y))=c(y,F_\alpha(x,y))=T$ and $\alpha\not\in T$.  
In particular, setting
\[\Psi_x(\alpha,i)=\left\{\begin{array}{cc}
     \Phi_x(\alpha,i)&\alpha\in S^*  \\
     0&\alpha\not\in S^*
\end{array}\right.\]
is a type II trivialization of $\Phi\upharpoonright A$. 

The general case will take a bit more work. 
We will make use of formal expressions; denote $\mathrm{Free}(Y)$ as the free abelian group over $Y$ and denote the basis element corresponding to $y$ as $e(y)$. 
We will define formal expressions for $1\leq s\leq n$, $\rho\in A^{s}$ and $\alpha\not\in S$.  
We first outline how what we construct fits into the broader framework. 
Note that for each $\mathbf{x}\in X^n$, $d\Phi$ provides maps $E_{\Phi}^\alpha\colon\mathrm{Free}(\{f\in X\mid \bigwedge\mathbf{x}(\alpha) \leq f(\alpha)\}^{n+1})\to G_{\alpha,\bigwedge\mathbf{x}(\alpha)}$ by interpreting $e(\mathbf{y})$ as $d\Phi(\mathbf{y})(\alpha)$ and using the relevant structure maps. 
Eventually, we will set $\Psi_\mathbf{x}(\alpha)=E_\Phi^\alpha(A_n^\alpha(\mathbf{x}))$ for some formal expression $A_n^\alpha(\mathbf{x})$ to be defined and show that $\Psi$ is a (type II) trivialization of $\Phi\upharpoonright A$ outside of $S^*$. 
That is, $\Psi_{\mathbf{x}}$ is finitely supported for each $\mathbf{x}\in A^n$ and whenever $\mathbf{y}\in A^{n+1}$ and $\alpha\in\kappa\setminus S^*$, $d\Psi_{\mathbf{y}}(\alpha)=d\Phi_{\mathbf{y}}(\alpha)$. 
We will then have that
\[\Psi'_{\mathbf{x}}(\alpha)=\left\{\begin{array}{cc}
     \Phi_{\mathbf{x}}(\alpha)&\alpha\in S^*  \\
     \Psi_{\mathbf{x}}(\alpha)&\text{otherwise} 
\end{array}\right.\]
is a type II trivialization of $\Phi$.

We extend the values of $F$ from sets to sequences by declaring $F(\mathbf{x})=F(\operatorname{range}(\mathbf{x}))$. 
The first type of formal expressions will be, for each $\mathbf{x}\in A^s$, $\alpha\not\in S$, and $1\leq s\leq n$,
\[A_s^\alpha(\mathbf{x})\in\mathrm{Free}(\{f\in X\mid \bigwedge\mathbf{x}(\alpha) \leq f(\alpha)\leq F_\alpha(\rho)(\alpha)\}^{s+1}).\]
The second two will be, for $\mathbf{x}\in A^{s+1}$ and $\alpha\not\in S$, expressions
\[C_s^\alpha(\mathbf{x})\in\mathrm{Free}(\{f\in X\mid \bigwedge\mathbf{x}(\alpha) \leq f(\alpha)< F_\alpha(\mathbf{x})(\alpha)\}^{s+1})\]
\[S_s^\alpha(\mathbf{x})\in\mathrm{Free}(\{f\in X\mid \bigwedge\mathbf{x}(\alpha) \leq f(\alpha)\leq F_\alpha(\tau)(\alpha)\}^{s+1}).\]
The idea is that $A_s^\alpha$ is the $s$-stage approximation to a trivialization, $S^\alpha_s$ is the boundary of $A_{s+1}^\alpha$ with a sign, and $C_s^\alpha$ is the error between the boundary of $A^\alpha_s$ and the boundary of $\Phi$. 

We also will make use of the operations
\[d\colon \mathrm{Free}(\{f\in X\mid \bigwedge\rho(\alpha) \leq f(\alpha)\}^{s+1})\to \mathrm{Free}(\{f\in X\mid \bigwedge\rho(\alpha) \leq f(\alpha)\}^{s})\]
given by $de(\mathbf{f})=\sum_i(-1)^ie(\mathbf{f}^i)$.
Also, given $g\in\omega^\kappa$, we define a map
\[-*g\colon \mathrm{Free}(\{f\mid\bigwedge\rho(\alpha)\leq f(\alpha)<g(\alpha)\}^{s})\to \mathrm{Free}(\{f\mid\bigwedge\rho(\alpha)\leq f(\alpha)\leq g(\alpha)\}^{s+1})\]
on the standard basis by setting $e(f_0,...,f_{s-1})*g=e(f_0,...,f_{s-1},g)$. 
These two operations satisfy the relation 
\[d(x*y)=d(x)*y+(-1)^sx.\]
As the base case $A_1^\alpha$, we set $A_1^\alpha=0$. 
In general, we define
\[C_s^\alpha(\mathbf{x})=e(\mathbf{x})-\sum_{i<n+1}(-1)^iA_s^\alpha(\mathbf{x}^i)\]
\[A_{s+1}^\alpha(\mathbf{x})=(-1)^{s+1}C_s^\alpha(\mathbf{x})*F_\alpha(\mathbf{x})\]
\[S_s^\alpha(\mathbf{x})=d\left(C_s^\alpha(\mathbf{x})*F_\alpha(\mathbf{x})\right)=(-1)^{s+1}d A^\alpha_{s+1}.\]
The following lemma is the key point in ensuring our approximations converge to a trivialization. 
Recall that $\tau^{\bb{s+1}}$ is the set of maximal strings on $\tau$ as defined in Definition \ref{str}. 

\begin{claim} \label{C=S}
$C_s^\alpha(\mathbf{x})$ is of the form $(-1)^{s+1}S^\alpha_s(\mathbf{x})$ plus terms of the form $e(F^*_\alpha(\vec\sigma))$ for $\vec\sigma\in \operatorname{range}(\mathbf{x})^{\bb{s+1}}$. 
\end{claim}

\begin{proof}
Observe that $S_s^\alpha(\mathbf{x})$ can be rewritten as follows: 
\[\begin{aligned}
&S^\alpha_s(\mathbf{x})&&=d(C_s^\alpha(\mathbf{x}))*F_\alpha(\mathbf{x})+(-1)^{s+1}C_s^\alpha(\mathbf{x})\\
&&&=d(e(\mathbf{x}))*F_\alpha(\mathbf{x})-d\left(\sum_{i<n+1}(-1)^iA_s^\alpha(\mathbf{x}^i)\right)*F_\alpha(\mathbf{x})+(-1)^{s+1}C_s^\alpha(\mathbf{x})\\
&&&=\sum_{i<s+1}(-1)^ie(\mathbf{x}^i,F_\alpha(\mathbf{x}))-\sum_{i<n+1}(-1)^id[A^\alpha_s(\mathbf{x}^i)]*F_\alpha(\mathbf{x})+(-1)^{s+1}C^\alpha_s(\mathbf{x}).
\end{aligned}\]
For convenience, we write 
\[(*)=\sum_{i<s+1}(-1)^ie(\mathbf{x}^i,F_\alpha(\mathbf{x}))-\sum_{i<n+1}(-1)^id[A^\alpha_s(\mathbf{x}^i)]*F_\alpha(\mathbf{x})\]

The proof proceeds by induction on $s$ to show that $(*)$ is a sum of terms of the form $e(\langle F_\alpha(a_i)\mid 1\leq i\leq s+1\rangle)$ for $\vec a\in\operatorname{range}(\mathbf{x})^{\bb{s+1}}$. 
The $s=1$ case is immediate: $A_1=0$ and the two terms on the left hand sum are of the appropriate form. 

Now assume that $s>1$ and the induction hypothesis holds for $s-1$. 
By definition of $S_{s-1}^\alpha$, for $s>1$ $(*)$ is equal to
\[\sum_{i<s+1}(-1)^ie(\mathbf{x}^i,F_\alpha(\mathbf{x}))-(-1)^s\sum_{i<s+1}(-1)^iS^\alpha_{s-1}(\mathbf{x}^i)*F_\alpha(\mathbf{x}).\]
By the inductive hypothesis, $(-1)^sS^\alpha_{s-1}(\mathbf{x}^i)*F_\alpha(\mathbf{x})$ is of the form $C_{s-1}^\alpha(\mathbf{x}^i)*F_\alpha(\mathbf{x})$ plus terms of the form $e(F^*(\vec a))$ for $\vec a\in[\operatorname{range}(\mathbf{x})]^{\bb {s+1}}$, as $e(\{F_\alpha^*(a_i)\mid 1\leq i\leq s\})*F_\alpha(\mathbf{x})$ is just appending $F_\alpha(\mathbf{x})$ to a string. 
Then $(*)$ reduces to terms of the form $e(\langle F_\alpha(a_i)\mid i\leq i\leq s+1\rangle)$ plus
\[\begin{aligned}
& \sum_{i<s+1}(-1)^{i} e\left(\mathbf{x}^{i}, F_\alpha({\mathbf{x}})\right)-\sum_{i<s+1}(-1)^{i}\left[\mathcal{C}_{s-1}^\alpha\left(\mathbf{x}^{i}\right) * F_\alpha({\mathbf{x}})\right] \\
=& \sum_{i<s+1}(-1)^{i} e\left(\mathbf{x}^{i}, F_\alpha({\mathbf{x}})\right)-\\&\sum_{i<s+1}(-1)^{i}\left[e\left(\mathbf{x}^{i}, F_\alpha({\mathbf{x}})\right)-\sum_{j<s}(-1)^{j} \mathcal{A}_{s-1}^\alpha\left(\left(\mathbf{x}^{i}\right)^{j}\right) * F_\alpha({\mathbf{x}})\right] \\
=& \sum_{i<s+1} \sum_{j<s}(-1)^{i+j} \mathcal{A}_{s-1}^\alpha\left(\left(\mathbf{x}^{i}\right)^{j}\right) * F_\alpha({\mathbf{x}}).
\end{aligned}\]
The key observation is that for any term with $i\leq j$, the corresponding term with $i'=j+1$ and $j'=i$ is the same but with opposite sign, thus showing that 
\[\sum_{i<s+1} \sum_{j<s}(-1)^{i+j} \mathcal{A}^\alpha_{s-1}\left(\left(\mathbf{x}^{i}\right)^{j}\right) * F_\alpha({\mathbf{x}})=0.\]

\end{proof}
As stated above, set $\Psi_\mathbf{x}(\alpha)=E_\Phi^
\alpha(A_n^\alpha(\mathbf{x}))$ for $\alpha\not\in S^*$ and $\mathbf{x}\in A^n$. 
If $\alpha\not\in S^*$ then for each $\mathbf{y}\in A^{n+1}$,
\[\begin{aligned}
    d\Psi_{\mathbf{y}}(\alpha)&=\sum_{i=0}^n(-1)^i\Psi_{\mathbf{y}^i}(\alpha)\\
    &=\sum_{i=0}^n(-1)^iE_\Phi^\alpha(A_n^\alpha(\mathbf{y}^i))\\
    &=E^\alpha_\Phi\left(\sum_{i=0}^n(-1)^iA_n^\alpha(\mathbf{y}^i)\right)\\
    &=E^\alpha_\Phi(e(\mathbf{y})-C^\alpha_{n+1}(\mathbf{y}))\\
    &=E^\alpha_\Phi(e(\mathbf{y}))-E^\alpha_\Phi(C^\alpha_{n+1}(\mathbf{y}))\\
    &=d\Phi_\mathbf{y}(\alpha)-E^\alpha_\Phi(C^\alpha_{n+1}(\mathbf{y}))
\end{aligned}\]
where the first line is the definition of $d$, the second line is the definition of $\Psi$, the third and fifth are that $E^\alpha_\Phi$ is a homomorphism, the fourth is the definition of $C^\alpha_{n+1}$, and the last line is the definition of $E^\alpha_\Phi$. 
By Claim \ref{C=S}, $C^\alpha_{n+1}(\mathbf{y})$ is of the form $(-1)^{s+1}S_{n+1}^\alpha(\mathbf{y})$ plus terms of the form $e(F^*_\alpha(\vec\sigma))$ for $\vec\sigma\in\operatorname{range}(\mathbf{y})^{\bb{n+1}}$. 
Note that $E_\Phi^\alpha(S_{n+1}^\alpha(\mathbf{y}))=d^2\Psi_{\mathbf{y}^\frown F_\alpha(\mathbf{y})}(\alpha)=0$. 
Moreover, since $c_\Phi\circ F_\alpha^*$ is constant with value $T$ and $\alpha\not\in T$, whenever $\vec\sigma\in\operatorname{range}(\mathbf{y})^{\bb{n+1}}$, we have $E^\alpha_\Phi(e(F^*_\alpha(\vec\sigma)))(\alpha)=0$. 
In particular, for $\alpha\not\in S^*$, $d\Phi_{\mathbf{y}}=d\Psi_{\mathbf{y}}$.

If we can show that each $\Psi_{\mathbf x}$ is finitely supported, we will complete the proof as 
\[\Psi_{\mathbf{x}}'(\alpha)=\left\{\begin{array}{cc}
     \Psi_{\mathbf{x}}(\alpha)&\alpha\not\in S^*  \\
     \Phi_{\mathbf{x}}(\alpha)&\alpha\in S^* 
\end{array}\right.\]
will be a type II trivialization of $\Phi$. 
It follows from an easy induction that $A_s^\alpha(\rho)$ is always a sum of terms of the form $e(a_0\cup\{F_\alpha(a_i)\mid 1\leq i\leq m\})\}$ for $\vec a$ a string on $[\mathbf{x}]^{\leq n}$. 
Since there are only finitely many such strings and there are only finitely many values of $c_\Phi$ for each string by choice of the $F_\alpha$, $\Psi_{\mathbf{x}}$ will have the desired finite support.
\end{proof}
\subsection{Trivialization Everywhere}
We will now complete the proof of our main results. 
Let $\lambda_0=\beth_1(\kappa)^+$ and $\lambda_{n+1}=\sigma(\lambda_n^+,2n+1)$; observe that $\sup_n(\lambda_n)=\beth_\omega(\kappa)$ and each $\lambda_n$ is $<\kappa$ inaccessible.  

\begin{lem}
Suppose that $\chi\geq \beth_\omega(\kappa)$ and let $\mathbb{P}=\operatorname{Fn}(\chi\times\kappa,\omega)$. 
The following holds in $V^\mathbb{P}$:
Suppose that $\mathcal{I}$ is an injective $\Om_{\kappa}$ system and $\Phi$ is an $n$-coherent family corresponding to $\mathcal{I}$ indexed by a set $X\subseteq\omega^\kappa$ containing at least $\lambda_n$ many of the Cohen sequences added by $\mathbb{P}$. 
Then $\Phi$ is $n$-trivial. 
\end{lem}
\begin{proof}
The proof proceeds by induction on $n$. 
By Lemma \ref{large_triv}, there is an $A\subseteq X$ containing at least $\lambda_{n-1}$ many Cohen sequences such that $\Phi\upharpoonright A$ is $n$-trivial. 
By Lemma \ref{propagate}, the inductive hypothesis, and Proposition \ref{unbdd}, $\Phi$ is $n$-trivial.
\end{proof}
We obtain Theorem \ref{add} as an immediate corollary.

\section{Conclusion}
We conclude by noting several questions related to the work here. 
One conspicuous example is the following.
\begin{question}
Is it consistent that derived limits of $\Om_\kappa$ systems are additive for every $\kappa$? 
\end{question}
The answer is yes in the case of $\lim^1\mathbf{A}_\kappa$, as shown in \cite[Theorem 5.1]{SHDLST}; in fact, it is implied by $\lim^1\mathbf{A}=0$. 
We may ask whether this holds for higher values of $n$.
\begin{question} \label{imply_quest}
Does $\lim^n\mathbf{A}=0$ imply $\lim^n\mathbf{A}_\kappa=0$ for every cardinal $\kappa$?
\end{question}
The author has been informed through personal communication of a forthcoming negative solution by Jeffrey Bergfalk and Matteo Casarosa to Question \ref{imply_quest} for all $n>0$ and consequently to Questions \ref{add_ques} and \ref{coh}.

Worth noting is that the forcings involved in this paper necessarily increase the size of the continuum; in particular, the case of $\kappa=2^{\aleph_0}$ is unclear.
\begin{question}
Is it consistent that $\lim^n\mathbf{A}_{2^{\aleph_0}}=0$ for all $n>0$?
\end{question}

There are two questions about $\Om_\kappa$ systems which serve as extension from the questions about $\mathbf{A}_\kappa$ to all $\Om_\kappa$ systems.
The first is on additivity.
\begin{question} \label{add_ques}
Does additivity of $\lim^n$ for $\Om_\omega$ systems imply additivity of $\Om_\kappa$ systems for all uncountable $\kappa$?
\end{question}
The second is about coherent families. 
\begin{question} \label{coh}
Suppose that every $n$-coherent family is $n$-trivial for all $\Om_\omega$ systems. 
Must the same hold for all uncountable $\kappa$?
\end{question}
At the moment, we believe the proof of \cite[Theorem 5.1]{SHDLST} will generalize to show the answer to Question \ref{coh} is yes if $n=1$. 

We may also ask about how large the continuum must be for additivity of derived limits, both for $\mathbf{A}$ and for arbitrary $\Om_\omega$ systems.
\begin{question} \label{minA}
What is the minimum value of the continuum compatible with $\lim^n\mathbf{A}=0$ for every $n>0$?
\end{question}
\begin{question} \label{min}
What is the minimum value of the continuum compatible with additivity of derived limits for $\Om_\omega$ systems?
\end{question}
Note that if $GCH$ holds in the ground model, then after forcing with $\operatorname{Add}(\omega,\beth_\omega)$ the continuum will be of size $\aleph_{\omega+1}$, giving an upper bound. 
In light of the results of \cite{NVHDL} and \cite{NVHDLWS}, of particular interest is when $\mathfrak{b}<\mathfrak{d}$ and there are no unbounded $\mathfrak{d}$ chains, for instance in the Mitchell and Miller models.
A related question, the answer to which is yes when $n=1$, is the following. 
\begin{question} \label{om_n_quest}
Suppose $X$ is a $<^*$-increasing $\omega_n$ sequence. Is $\lim^n\mathbf{A}\upharpoonright X\neq0$?
\end{question}
A theorem from the theory of Hausdorff gaps implies the answer is yes when $n=1$; see \cite[96-98]{BE}; a special case when $X$ is an $\aleph_1$ scale appears as \cite[Theorem 2.4]{DSV}. 
Question \ref{om_n_quest} is likely very difficult and would open many doors. 
A positive answer would remove the requirement of weak $\diamondsuit$ sequences from \cite{NVHDL} since if $\mathfrak{b}=\mathfrak{d}=\kappa$ then there is a $<^*$-increasing $\omega_n$ sequence which is $<^*$ cofinal in $\omega^\omega$; a nontrivial $n$-coherent family defined along such a scale induces a nontrivial coherent family on all of $\omega^\omega$. 
It seems plausible that a positive answer would also remove the requirement of weak $\diamondsuit$ from \cite{NVHDLWS} as well, though the $<^*$-increasing $\omega_n$ sequence defined there is (importantly) not $<^*$-cofinal. 

Finally, we conclude by asking how far our results on additivity of strong homology can be extended. 

\begin{question}
    On what class of spaces can strong homology be additive and have compact supports?
\end{question}
The classes of Polish spaces and $\sigma$-compact Polish spaces are of particular interest.

\end{document}